\documentclass[journal]{IEEEtran}
\newenvironment{proof}{\emph{Proof:}}{\hspace{\stretch{1}}\rule{1ex}{1ex}}
\newtheorem{definition}{Definition}
\newtheorem{lemma}{Lemma}
\newtheorem{remark}{Remark}
\newtheorem{theorem}{Theorem}

\usepackage{amsmath,amssymb,amsfonts}
\usepackage{algorithmic}
\usepackage{graphicx}
\usepackage{textcomp}
\usepackage{cite}
\usepackage{amsfonts}
\usepackage{amssymb}
\usepackage{color}
\allowdisplaybreaks

\ifCLASSINFOpdf
\else
\fi
\begin{document}
\title{Stability analysis of random nonlinear systems with time-varying delay
and its application}

\author{Liqiang~Yao, Weihai~Zhang${}^*$% <-this % stops a space
\thanks{${}^*$  Corresponding author.}
\thanks{This work was supported by National Natural Science Foundation
of China (Nos: 61573227, 61633014), SDUST Research Fund (No.
2015TDJH105), Research Fund for the Taishan Scholar Project of
Shandong Province of China. }
\thanks{Liqiang~Yao is with the College of Electrical Engineering and Automation, Shandong
University of Science and Technology, Qingdao,  Shandong Province,
266590, P. R. China and the School of Mathematics and Information
Science, Yantai University, Yantai, Shandong Province, 264005, P. R. China.}% <-this % stops a space
\thanks{Weihai~Zhang is with the College of Electrical Engineering and Automation, Shandong
University of Science and Technology, Qingdao,  Shandong Province, 266590,
P. R. China. (Corresponding author: w\_hzhang@163.com).}% <-this % stops a space
\thanks{Manuscript received April 19, 2005; revised August 26, 2015.}}


\markboth{Journal of \LaTeX\ Class Files
%,~Vol.~14, No.~8, August~2015
}%
{Shell \MakeLowercase{\textit{et al.}}: Bare Demo of IEEEtran.cls for IEEE Journals}



\maketitle

\begin{abstract}
This paper studies a class of random nonlinear systems with
time-varying delay,  in which the $r$-order moment ($r\geq1$)
of the random disturbance is finite. Firstly, some general conditions
are proposed to guarantee the existence and uniqueness of the global
solution to random nonlinear time-delay systems.
Secondly, some definitions and criteria on noise-to-state stability in
the moment sense and in probability sense are given by Lyapunov method
respectively. Finally, two regulation controllers are constructed
respectively for two corresponding random nonlinear time-delay
systems and the effectiveness of two proposed  recursive procedures are
demonstrated by two simulation examples.
\end{abstract}


% \fi
\IEEEpeerreviewmaketitle

\section{Introduction}\label{S1}
Time-delay is widespread in many actual systems
such as network transmission systems, hydraulic systems and power systems.
The time-delay phenomenon is often caused by  the
inherent characteristics of physical systems such as communication system
and mechanical drive system. In addition, the devices of practical
systems (for instance, controllers and actuators) need time to complete
operations during the running process of systems, which  inevitably
makes the time-delay phenomenon. The existence of time-delay phenomenon can
affect system performance seriously and brings difficulties for system
analysis and synthesis. Nevertheless, time-delay phenomenon
sometimes can be used to improve the control performance of some systems such as
repetitive control systems. Therefore, it is of great significance to study
how to eliminate and utilize time-delay phenomena. Over the past few decades,
many scholars  have committed themselves to studying time-delay
systems in theory and engineering applications ([\ref{Malek-Zavarei1987}-\ref{Lakshmanan2010}]).

On the other hand, many practical systems in engineering, which are
often subject to random disturbance from external environment, are
modeled as stochastic nonlinear systems or random nonlinear
systems ([\ref{Soong1973}-\ref{Wu2015}]).
The difference lies in the random  disturbance which appeared in
the former is a white noise process (i.e., the derivative
of Wiener process), while the
random disturbance in the latter is a stationary stochastic
process. Since the white noise has infinite bandwidth,
some actual systems are more suitable to be modeled as
random nonlinear systems from the energy view.

Along with  the work of Khas'minskii, Krsti\'{c} and Deng, Mao
and other scholars in stochastic control field, the remarkable
development has been achieved (please refer to
[\ref{Khasminskii1980}], [\ref{Krstic1998}-\ref{Mao2007}] and the
references therein). This  promotes the study of stochastic
nonlinear time-delay systems. [\ref{Fu2005}] designed an
output-feedback controller by Lyapunov-based recursive method such
that a class of stochastic lower triangular systems with constant
delay only appeared in drift term is exponentially stable. For
general stochastic nonlinear time-delay systems, [\ref{Liu2008}]
proposed some conditions to guarantee that systems solution exists
and is unique, and gave the global stability criteria in probability
sense. Besides, the adaptive output-feedback stabilization controller
for a class of stochastic lower triangular  systems with time-delay is
constructed in [\ref{Liu2008}]. By the stability results in
[\ref{Liu2008}], [\ref{Liu2011}] constructed an output-feedback
controller to achieve globally asymptotically stable in probability sense
for stochastic high-order nonlinear time-delay systems satisfying
some assumptions  and [\ref{Liu2013}] studied stabilization control
for a class of stochastic upper-triangular nonlinear systems
with time-delay by state feedback approach. However, for random nonlinear
systems, most findings in existing literature (such as
[\ref{Wu2015}, \ref{ZhangD2016}-\ref{Yao2018}]) didn't investigate
the influence of time-delay phenomenon on systems dynamics behavior.
The corresponding results about random nonlinear time-delay systems
also have not been reported in existing literature. This motivates
us to focus on random nonlinear time-delay systems.

Inspired by [\ref{Wu2015}] and [\ref{Liu2008}], this paper considers the
stability of random time-delay systems and regulation control
problems for random  nonlinear feedback time-delay systems.
The main contributions of this paper include three aspects:\\
(1) For random  time-varying delay systems with some general conditions,
this paper analyzes the existence and uniqueness of global solution. \\
(2) Some definitions and criteria on noise-to-state stability
(including in the moment sense   and in probability sense )
are proposed based on Lyapunov function approach. \\
(3) Applying the obtained stability results,
two classes of regulation control problems are studied
by constructing different Lyapunov-Krasovskii functionals, respectively.

This paper is organized as follows. Section II investigates the
existence and uniqueness of global solution for random systems with
time-delay. Some definitions and noise-to-state
stability criteria  are proposed in Section III. As
applications, two kinds of regulation problems are discussed in Section IV and
Section V, respectively. In Section VI, two corresponding simulation
examples illustrate the feasibility and effectiveness of proposed
two controller design procedures. Section VII gives the conclusions
of this paper.

{\bfseries Notions:}
%The following notations  will be used in this paper.
$|x|$ and $x^{T}$ represent the usual Euclidean norm
and the transpose of vector $x$, respectively.
$||A||$ denotes the 2-norm of matrix A and  $||A||_F$ stands for the
Frobenius norm of matrix A.
$\mathbb{R}_+$ and $\mathbb{R}^n$ represent the set of nonnegative real numbers
and the real $n$-dimensional space, respectively.
$C([t-\tau, t]; \mathbb{R}^n)$ represents the space of continuous functions
$q$ from $[t-\tau, t]$ to $\mathbb{R}^n$ with the norm
$|q_t|=\sup_{-\tau \leq \theta \leq 0} |q(t+\theta)|$ for $\tau>0$ and $t\geq0$.
The set of functions with continuous $i$-th partial derivative is denoted as $C^i$,
and the function $W(t, x(t))\in
C^{1,1}([t_0-\tau, \infty) \times \mathbb{R}^n ; \mathbb{R}_+)$ means that
$W(t, x(t))$ are $C^1$ in $t$ and  $C^1$ in $x$.
$C_{\mathcal{F}_{t_0}}^b([t_0-\tau, t_0]; \mathbb{R}^n)$ denotes the family of all
$\mathcal{F}_{t_0}$-measurable bounded $C([t_0-\tau, t_0]; \mathbb{R}^n)$-valued
random variable $\varphi=\{\varphi(\theta):t_0-\tau \leq \theta \leq t_0\}$
with $ t_0 \geq 0$. $\gamma(t)\in\mathcal{K}$ means that the function $\gamma(t)$
defined on $\mathbb{R}_+$ is strictly increasing, continuous and vanish at origin;
$\bar{\gamma}(t)\in\mathcal{K}_\infty$ implies that
$\bar{\gamma}(t)\in\mathcal{K}$  and $\bar{\gamma}(t)$ is
unbounded; $\beta(s,t) \in \mathcal{KL}$ denotes that
$\beta(s,t) \in \mathcal{K}$ for each fixed $t$, and
$\lim_{t \rightarrow \infty}\beta(s,t)=0$ holds for each fixed $s$.
Function $h(s)$ is convex on $D$, if  for
any $s_1, s_2 \in D$, $h(s)$ satisfies $h((s_1+s_2)/2)\leq (h(s_1)+ h(s_2))/2$.

\section{Preliminaries}\label{S2}
In this section, one discusses the existence and uniqueness of solution to
random differential delay equations (RDDEs) that are a special class of
random functional differential equations (RFDEs). So, one firstly discusses the
the existence and uniqueness of solution to RFDEs and then analyze
the existence and uniqueness of solution to RDDEs.
\subsection{The existence and uniqueness of solution to RFDEs}
Consider the following random functional differential equation
\begin{equation}\label{eq21}
\dot{x}(t)=f(t, x_t)+g(t, x_t)\xi(t), ~~t\geq t_0,
\end{equation}
with the initial data
$x_{t_0}=\varphi=\{\varphi(t_0+\theta): -\tau \leq \theta \leq 0\}
\in C_{\mathcal{F}_{t_0}}^b([t_0-\tau, t_0]; \mathbb{R}^n)$,
where $x_t=\{x(t+\theta): \theta \in [-\tau, 0]\}$
is a $C([t-\tau, t]; \mathbb{R}^n)$ -valued stochastic process.
The state is $x(t)\in \mathbb{R}^{n}$,
$\xi(t)\in \mathbb{R}^{m} $ represents a  piecewise continuous and $\mathcal{F}_t$-adapted
stochastic process defined on the complete filtered probability space
$(\Omega, \mathcal{F}, \{\mathcal {F}_t\}_{t\geq t_0}, P)$,
and it has finite $r$-order moment (i.e., $\sup_{t\geq t_0} E|\xi(t)|^r<K$
with $r\geq 1$ and $K$ being positive constants). For $t_0\leq t \leq T <\infty$, functions
$f: [t_0, T] \times C([t-\tau, t]; \mathbb{R}^n)\rightarrow \mathbb{R}^n $ and
$g: [t_0, T] \times C([t-\tau, t]; \mathbb{R}^n)\rightarrow \mathbb{R}^{n\times m}$
are Borel measurable, piecewise continuous in $t$ and locally Lipschitz continuous
in $x_t\in \mathbb{R}^n $, moreover, $f(t, 0)$ and $g(t, 0)$ are bounded.
\begin{definition}\label{def21}
A solution $x(t)$ to equation (\ref{eq21}) on $[t_0-\tau, T]$ with  initial value
$x_{t_0}=\varphi$ is an  $\mathbb{R}^n$-valued stochastic process
and satisfies that \\
(i) $x(t)$ and $\{x_t\}_{t_0 \leq t \leq T}$ are continuous and $\mathcal{F}_t$-adapted.\\
(ii) $x(t)=x_{t_0} + \int_{t_0}^t f(s, x_s)ds + \int_{t_0}^t g(s, x_s)\xi(s)ds$
holds almost surely for $t_0\leq t\leq T$.
\end{definition}

The uniqueness of solution to equation (\ref{eq21}) is in almost sure sense, that is
to say, a solution $x(t)$ is called to be unique if
$P\{x(t)=\bar{x}(t), t_0-\tau \leq t\leq T\}=1$, where
$\bar{x}(t)$ is another solution to equation (\ref{eq21}).
For any $ T>t_0-\tau$, if there exists a unique solution
to equation (\ref{eq21}) on $[t_0-\tau, T]$ , equation (\ref{eq21})
has a unique solution on $[t_0-\tau, \infty)$.

In order to establish the sufficient conditions, which ensure that
the global solution to equation (\ref{eq21}) exists and is unique,
one introduces the first exit time  from a region
$U_k=\{x: |x|<k\}$ and its limit.
Let $\rho_k=\inf\{t\geq t_0 : |x(t)| \geq k\}$
and $\rho_\infty=\lim_{k\rightarrow\infty}\rho_k$ almost surely,
where $\rho_k$ are stopping times,
$ \inf\phi=\infty $, and $k \geq 1$ is any integer.

\begin{lemma}\label{lem2a1}
For equation (\ref{eq21}), if there is a
function $V(t,x(t))\in C^{1,1}([t_0-\tau, \infty)\times \mathbb{R}^n ; \mathbb{R}_+)$
and constants $c_0$ and $d_0>0$ such that for any integer $k \geq 1$ and  $t\geq t_0$,
\begin{gather}
\label{lem2a1a}\lim_{k\rightarrow\infty}\inf_{|x|\geq k}V(t,x)=\infty,\\
\label{lem2a1b} EV(\rho_k \wedge t, x(\rho_k \wedge t))\leq d_0 e^{c_0t},
\end{gather}
then there exists a unique solution  on $[t_0-\tau, \infty)$ to equation (\ref{eq21}) .
\end{lemma}

\begin{proof}
For every integer $k \geq 1$, one defines the following truncation functions
\begin{align*}
f_k(t, x_t)=
\begin{cases}
f(t, x_t) &\textmd{if}~~|x_t|\leq k \\
f(t, kx_t/|x_t|) &\textmd{if}~~|x_t|>k \\
\end{cases}, \cr
g_k(t, x_t)=
\begin{cases}
g(t, x_t) &\textmd{if}~~|x_t|\leq k \\
g(t, kx_t/|x_t|) &\textmd{if}~~|x_t|>k \\
\end{cases}.
\end{align*}
Then $f_k(t, x_t)$ and $g_k(t, x_t)$ satisfy the Lipschitz condition.
Adopting a similar way as Lemma 3 in [\ref{Wu2015}], one can
prove that there exists a unique $\mathcal {F}_t$-adapted solution $x_k(t)$ satisfying
\begin{equation}\label{eq22}
\dot{x}_k(t)=f_k(t, x_{t,k})+g_k(t, x_{t,k})\xi(t), ~~t\geq t_0,
\end{equation}
with the initial value $x_{t_0, k}=\varphi \in C_{\mathcal{F}_{t_0}}^b([t_0-\tau, t_0]; \mathbb{R}^n)$,
the details of proof is omitted here.
It is clear that
\begin{equation}\label{eq23}
x_k(t)=x_{k+1}(t), ~~~\forall t\in [t_0, \rho_k).
\end{equation}
Define $x(t)=x_k(t), t\in [t_0, \rho_k), k\geq 1$, % where $\rho_0=t_0$,
then it follows from (\ref{eq22}) and (\ref{eq23}) that for any
$t\geq t_0$ and $k\geq 1$,
\begin{align*}
x(t \wedge \rho_k)
=&\varphi(t_0) + \int_{t_0}^{t \wedge \rho_k}[ f(s, x_{s,k})+g(s, x_{s,k})\xi(s)]ds.
\end{align*}

First, one considers the case of $\rho_\infty<T<\infty$, then
\begin{equation*}
\limsup_{t\rightarrow \rho_\infty}|x(t)|
\geq\limsup_{k\rightarrow \infty}|x(\rho_k)|
=\limsup_{k\rightarrow \infty}|x_k(\rho_k)|=\infty,
\end{equation*}
which means that equation (\ref{eq21}) has a maximal solution $x(t), t\in[t_0, \rho_\infty)$,
where $\rho_\infty$ is often called the explosion time.
The uniqueness of $x_k(t)$ on $[t_0, \rho_\infty)$ can lead to the uniqueness of
the solution $x(t)$ on $[t_0, \rho_\infty)$. By now, there exists a unique solution $x(t)$ to
equation (\ref{eq21}) on $[t_0-\tau, \rho_\infty)$ and
the  solution $x(t)$ is $\mathcal{F}_t$-adapted, which can be inferred from
the fact that $x_{k}(t)$ is $\mathcal{F}_t$-adapted.

Next, one begins to show $\rho_\infty=\infty $  almost surely.
In fact, if $\rho_\infty < \infty$ almost surely, then there exist
constants $T>0$ and $\varepsilon>0$ such that $P\{\rho_\infty\leq T\}>\varepsilon$.
Because $\lim_{k\rightarrow\infty}\rho_k=\rho_\infty$ almost surely,
there is a integer $k_0$ such that
\begin{equation}\label{eq24}
P\{\rho_k\leq T\}>\varepsilon,~~\forall k\geq k_0.
\end{equation}
For any fixed $k(k\geq k_0)$, let $t=T$ in (\ref{lem2a1b}), then one
has $EV(T\wedge\rho_k, x(T\wedge\rho_k))\leq d_0 e^{c_0T}$,
from which one also can obtain
\begin{equation}\label{eq25}
E[V(\rho_k, x(\rho_k))I_{\{\rho_k\leq T\}}]\leq d_0 e^{c_0T}.
\end{equation}

On the other hand, one defines
$\varrho_{k,T}=\inf\{V(t,x(t)):|x(t)|\geq k, t\in [t_0, T]\}$,
then $\lim_{k\rightarrow\infty}\varrho_{k,T}=\infty$  by (\ref{lem2a1a}).
It follow from (\ref{eq24}) and (\ref{eq25}) that
$d_0 e^{c_0T}\geq \varrho_{k,T} P\{\rho_k\leq T\}>\varepsilon \varrho_{k,T}$.
Since letting $k\rightarrow \infty$ yields a contradiction, one has
$\rho_\infty=\infty$ almost surely.
So, there is a unique global solution $x(t)$
on $[t_0-\tau, \infty)$ for  equation (\ref{eq21}).
\end{proof}

\begin{lemma}\label{lem2a2}
For  equation (\ref{eq21}), aussume that there are a function
$V(t,x(t))\in C^{1,1}([t_0-\tau, \infty)\times \mathbb{R}^n ; \mathbb{R}_+)$
and a constant $K_1>0$ such that
\begin{gather}
\label{lem2a2a}\lim_{|x|\rightarrow\infty}\inf_{t\geq t_0}V(t,x)=\infty,\\
%\label{lem2a2b}\dot{V}\leq K(1+V(t,x)+V(t+\theta, x(t+\theta))),
\label{lem2a2b}E\dot{V}(t,x)\leq K_1(1+E V(t,x)+E V(t+\theta,x(t+\theta))),
\end{gather}
where $-\tau \leq \theta \leq 0$ and $t\geq t_0$, then equation (\ref{eq21})
has a unique solution on $[t_0-\tau, \infty)$.
\end{lemma}

\begin{proof}
From the proof of Lemma \ref{lem2a1}, one knows that there is a unique
maximal local solution $x(t)$ on $[t_0-\tau, \rho_\infty)$ for equation (\ref{eq21}).
Therefore, one need to testify that  $\rho_\infty=\infty $ almost surely.
For any $k \geq 1$,
let $\varrho_k=\rho_\infty \wedge \inf\{t_0 \leq t < \rho_\infty: |x(t)| \geq k\}$.
It is certain that $\varrho_\infty \leq \rho_\infty$ and
$\varrho_\infty=\lim_{k\rightarrow\infty}\varrho_k$ almost surely.

By Fubini's theorem [Theorem 2.39, \ref{Klebaner2005}] and (\ref{lem2a2b}), one has
\begin{align}\label{lem2eq1}
&EV(t \wedge \varrho_k, x(t \wedge \varrho_k))\cr
\leq & \bar{H}(t)+K_1\int_{t_0}^{t} \sup_{t_0\leq v \leq s}
[EV(v \wedge \varrho_k, x(v \wedge \varrho_k))]ds\cr
\triangleq &  H^*(t).
\end{align}
where $t \geq t_0$ , $k \geq 1$, and
$\bar{H}(t)=EV(t_0, x(t_0))+K_1(t-t_0)+K_1\int_{t_0}^{t}EV(s+\theta, x(s+\theta))ds>0$.
Because $H^*(t)$ is a increasing function of $t$,
one can obtain the following inequality
from (\ref{lem2eq1})
\begin{align*}
&\sup_{t_0\leq v \leq t} [EV(v \wedge \varrho_k, x(v \wedge \varrho_k))]\cr
\leq & \bar{H}(t)+K_1\int_{t_0}^{t} \sup_{t_0\leq v \leq s}
[EV(v \wedge \varrho_k, x(v \wedge \varrho_k))]ds.
\end{align*}
%where $\bar{H}(t)=\bar{H}(t)+K_1t$.
By Gronwall's inequality [Theorem 1.20, \ref{Klebaner2005}], one gets
\begin{align*}
\sup_{t_0\leq v \leq t} [EV(v \wedge \varrho_k, x(v \wedge \varrho_k))]
\leq\bar{H}(t)e^{K_1(t-t_0)} .
\end{align*}
It is certain that
\begin{align}\label{lem2eq2}
EV(t \wedge \varrho_k, x(t \wedge \varrho_k))
\leq \bar{H}(t)e^{K_1t}, ~~\forall t \geq t_0.
\end{align}
On the other hand, for any $ R \geq 0$, one defines that
$\delta(R)=\inf_{R\leq|x|, t_0\leq t}V(t,x(t))$,
then  it is easy to show that $\delta(|x(t)|)\leq V(t,x(t))$ and
$\lim_{R\rightarrow\infty}\delta(R)=\infty$ which comes from (\ref{lem2a2a}).
It follows from (\ref{lem2eq2}) that
\begin{align*}
P\{\varrho_k\leq t\}\leq \frac{E\delta(|x(t \wedge \varrho_k)|)}{\delta(k)}
\leq \frac{\bar{H}(t)e^{K_1t}}{\delta(k)}.
\end{align*}
Letting first $k \rightarrow \infty$ and then $t \rightarrow \infty$,
$P\{\varrho_\infty < \infty \}=0$ can be obtained, that is,
$\varrho_\infty = \infty$ almost surely.
This means that $\rho_\infty = \infty$ almost surely.
Thus, equation (\ref{eq21}) has a unique solution
on $[t_0-\tau, \infty)$.
\end{proof}
\subsection{The existence and uniqueness of solution to RDDEs}
For the following random differential delay equations
\begin{align}\label{eq26}
\dot{x}=f(t,  x(t-\tau(t)), x)+g(t, x(t-\tau(t)), x)\xi(t),
\end{align}
with the initial value $x_{t_0}=\varphi=\{\varphi(t):t_0-\tau \leq t
\leq t_0\}\in C_{\mathcal{F}_{t_0}}^b([t_0-\tau, t_0];
\mathbb{R}^n)$. System state is $x(t)\in \mathbb{R}^{n}$, and the
random disturbance  $\xi(t)\in \mathbb{R}^{m} $ is in accord
with that of system (\ref{eq21}). Borel measurable function
$\tau(t)$ is defined on $[t_0, \infty)$ and takes values on $[0, \tau]$.
Borel measurable functions $f: [t_0, \infty) \times
\mathbb{R}^n\times \mathbb{R}^n\rightarrow \mathbb{R}^n $ and $g:
[t_0, \infty) \times \mathbb{R}^n\times \mathbb{R}^n\rightarrow
\mathbb{R}^{n\times m} $  satisfy
locally Lipschitz condition with respect to $x(t-\tau(t))$ and $x(t) $,
respectively. In addition, $f(t, 0, 0)$ and $g(t,
0, 0)$ are bounded.

%Since $x_t=\{x(t+\theta):-\tau \leq \theta \leq 0\}$,
One introduces
$F(t, \bar{x}_t)=f(t, \bar{x}_t(-\tau(t)), \bar{x}_t(0))=f(t, x(t-\tau(t)), x(t))$ and
$G(t, \bar{x}_t)=g(t, \bar{x}_t(-\tau(t)), \bar{x}_t(0))=g(t, x(t-\tau(t)), x(t))$,
then equation (\ref{eq26}) can be rewritten as
\begin{equation}\label{eq27}
\dot{x}=F(t, \bar{x}_t)+G(t, \bar{x}_t)\xi(t), ~~t\geq t_0.
\end{equation}
This implies that RDDEs are in fact a special class of RFDEs.
Therefore, one can obtain the following sufficient conditions such that system (\ref{eq26})
has a unique global solution.

\begin{lemma}\label{lem2b1}
For  system (\ref{eq26}), if there are a function
$V(t,x(t))\in C^{1,1}([t_0-\tau, \infty)\times \mathbb{R}^n ; \mathbb{R}_+)$
and constants $c_0$ , $d_0>0$ such that (\ref{lem2a1a}) and (\ref{lem2a1b})
hold for any integer $k \geq 1$ and $t\geq t_0$, then there is a unique solution
to system (\ref{eq26}) with initial value
$x_{t_0}=\varphi\in C_{\mathcal{F}_{t_0}}^b([t_0-\tau, t_0]; \mathbb{R}^n)$.
\end{lemma}

\begin{proof}
For system (\ref{eq26}),  since both $f(t, x(t-\tau(t)), x)$ and
$g(t,  x(t-\tau(t)), x)$ are Borel measurable functions and
satisfy locally Lipschitz condition with respect to $x(t-\tau(t))$ and
$x$, respectively. In other words, for any $R>0$, there is a positive constant $L_R$ such that
for $\forall x_1, x_2\in U_R, x_1\neq x_2$ and
$\forall x_1(t-\tau(t)), x_2(t-\tau(t))\in U_R, x_1(t-\tau(t))\neq x_2(t-\tau(t))$,
\begin{align}\label{eq28}
&|f(t, x_1(t-\tau(t)), x_1)-f(t, x_2(t-\tau(t)), x_2)|\cr
\leq &L_R(|x_1-x_2|+|x_1(t-\tau(t))-x_2(t-\tau(t))|),\cr
&||g(t, x_1(t-\tau(t)), x_1)-g(t, x_2(t-\tau(t)), x_2)||\cr
\leq & L_R(|x_1-x_2|+|x_1(t-\tau(t))-x_2(t-\tau(t))|).
\end{align}
Note that $|q_t|=\sup_{0 \leq \tau(t) \leq \tau} |q(t-\tau(t))|$,
then one has
\begin{align}\label{lem29}
&|x_1-x_2|\vee |x_1(t-\tau(t))-x_2(t-\tau(t))|\cr
\leq &|\bar{x}_{1,t}-\bar{x}_{2,t}|.
\end{align}
It follows from (\ref{eq28}) and (\ref{lem29}) that
\begin{align}\label{eq30}
&|F(t, \bar{x}_{1,t})-F(t, \bar{x}_{2,t})| \vee ||G(t, \bar{x}_{1,t})-G(t, \bar{x}_{2,t})||\cr
\leq & 2L_R|\bar{x}_{1,t}-\bar{x}_{2,t}|.
\end{align}
This means that equation (\ref{eq27}) satisfies the locally Lipschitz condition.
From Lemma \ref{lem2a1}, there is a unique solution on $[t_0-\tau, \infty)$ to
system (\ref{eq27}) if there are a positive function
$V(t,x(t))\in C^{1,1}([t_0-\tau, \infty)\times \mathbb{R}^n ; \mathbb{R}_+)$,
constant $c_0$ and positive constant $d_0$ satisfying (\ref{lem2a1a})
and (\ref{lem2a1b}).
Since equation (\ref{eq26}) is equivalent to equation
(\ref{eq27}), there exists a unique solution to system (\ref{eq26}).
\end{proof}

\begin{definition}\label{def22} [Definition 1, \ref{Wu2015}]
For any $\varepsilon>0$, if there is a positive constant $\epsilon$ such that
$P\{\sup_{t\geq t_0}|\phi(t)|>\epsilon\}\leq \varepsilon$ holds,
then the stochastic process $\phi(t)$ is  bounded in probability.
\end{definition}

\begin{lemma}\label{lem2b2}
For  system (\ref{eq26}), if there are a function
$V(t,x(t))\in C^{1,1}([t_0-\tau, \infty)\times \mathbb{R}^n ; \mathbb{R}_+)$
and a constant $K_2>0$ such that for any $t\geq t_0$,
$\lim_{|x|\rightarrow\infty}\inf_{t\geq t_0}V(t,x)=\infty$ and
$E\dot{V}(t,x)\leq K_2(1+EV(t,x)+EV(t-\tau(t), x(t-\tau(t)))$ hold,
then system (\ref{eq26}) has a unique solution.
\end{lemma}

\begin{remark}\label{rem21}
Since the proof of Lemma \ref{lem2b2} is similar to that of Lemma \ref{lem2a2},
the detailed proof of Lemma \ref{lem2b2} is omitted  here.
\end{remark}

\begin{lemma}\label{lem2b3}
For  system (\ref{eq26}), if there exist a function
$V(t,x(t))\in C^{1,1}([t_0-\tau, \infty)\times \mathbb{R}^n ; \mathbb{R}_+)$
and a positive constant $d_0$ such that for any integer $k \geq 1$, (\ref{lem2a1a}) and
\begin{align}\label{lem2b2eq}
EV(t\wedge\rho_k, x(t\wedge\rho_k))\leq d_0
\end{align}
hold, then system (\ref{eq26}) has a unique solution and this unique solution is bounded in probability.
\end{lemma}

\begin{proof}
According to Lemma \ref{lem2b1}, system (\ref{eq26}) has a unique global solution.
From (\ref{lem2b2eq}), one has
\begin{align*}
EV(t, x(t))=\lim_{k\rightarrow \infty}EV(\rho_k\wedge t, x(\rho_k\wedge t))\leq d_0.
\end{align*}
On the other hand, by Lemma 1.4 in [\ref{Khasminskii1980}], one gets
\begin{align*}
EV(t, x(t))&\geq \int_{\{\sup_{t\geq t_0}|x(t)|> r\}}V(t, x(t))dP\cr
&\geq P\{\sup_{t\geq t_0}|x(t)|> r\}\inf_{t\geq t_0,~|x(t)|> r}V(t, x(t)).
\end{align*}
Thus, there is a positive constant $r$  such that for any $\varepsilon>0$,
\begin{align*}
P\{\sup_{t\geq t_0}|x(t)|> r\}\leq \frac{d_0}{\inf_{t\geq t_0, ~|x(t)|> r}V(t, x(t))}<\varepsilon.
\end{align*}
This implies that the unique global solution of system (\ref{eq26}) is bounded in probability.
\end{proof}

\begin{lemma}\label{lem2b4} [\ref{Flett1980}]
$y(t)$, $k(t)$ and $w(t)$ are continuous functions on $[t_0, \infty)$ such that
$$\limsup_{\vartriangle t\rightarrow 0^+}
\frac{y(t+\vartriangle t)-y(t)}{\vartriangle t}
=D^{+}y(t)\leq k(t)y(t)+w(t)$$
holds for almost all $t\in[t_0, \infty)$.
Then,
$$y(t)\leq y(t_0)e^{\int_{t_0}^t k(s)ds}+\int_{t_0}^t e^{\int_{s}^t k(v)dv} w(s)ds, ~~\forall t\geq t_0.$$
\end{lemma}
\begin{lemma}\label{lem2b5} [\ref{Hardy1989}]
For functions $\chi_1, \chi_2, \cdots, \chi_n$ and $l\geq 1$, the following inequality holds
\begin{align*}
|\chi_1+\chi_2+\cdots+\chi_n|^l\leq n^{l-1}(|\chi_1|^l+|\chi_2|^l+\cdots+|\chi_n|^l).
\end{align*}
\end{lemma}
\section{Main results}\label{S3}
In this section, one discusses the noise-to-state stability of random
time-delay systems (\ref{eq26}) in  the moment sense and in probability
sense, respectively. First, for random nonlinear time-delay systems (\ref{eq26}),
one gives some definitions about the ultimate boundedness of system state and
noise-to-state stability.

\begin{definition}\label{def3a1}
If there exist functions $\gamma \in \mathcal{K}$ and $\beta\in \mathcal{KL}$ such that
$E|x(t)|^m \leq \beta(|\varphi|, t-t_0)+\gamma(\sup_{t_0 \leq l \leq t} E|\xi(l)|^r)$
holds for system (\ref{eq26}) with any given initial value
$\varphi\in C_{\mathcal{F}_{t_0}}^b([t_0-\tau, t_0]; \mathbb{R}^n)$,
where $m$ is a positive constant and
$|\varphi|=\sup_{t_0-\tau \leq s \leq t_0} |\varphi(s)|$,
then system (\ref{eq26}) is called to be noise-to-state stable
in the $\textbf{m}$-th moment sense(NSS-\textbf{m}-M).
\end{definition}

\begin{definition}\label{def3a2}
For any $\varepsilon>0$, if there exist functions
$\gamma \in \mathcal{K}$ and $\beta\in \mathcal{KL}$ such that
$P\{|x(t)|\leq \beta(|\varphi|, t-t_0)
+\gamma(\sup_{ t_0\leq l \leq t} E|\xi(l)|^r)\}\geq 1-\varepsilon$
holds for system (\ref{eq26}) with any given initial value
$\varphi\in C_{\mathcal{F}_{t_0}}^b([t_0-\tau, t_0]; \mathbb{R}^n)$,
where $|\varphi|=\sup_{t_0-\tau \leq s \leq t_0} |\varphi(s)|$,
then system (\ref{eq26}) is called to be noise-to-state stable
in probability sense (NSS-P).
\end{definition}

\begin{definition}\label{def3a3}
Let $m>0$ be  a constant. If there is a function $\gamma \in \mathcal{K}$ such that
$\lim_{t\to\infty} E|x(t)|^m \leq \gamma(K)$ holds,
then the state $x(t)$ of system (\ref{eq26}) is  ultimately bounded in the
$\textbf{m}$-th moment sense (UB-\textbf{m}-M).
\end{definition}

\begin{definition}\label{def3a4}
For any $\varepsilon>0$, if there is a function $\gamma \in \mathcal{K}$
such that $P\{\lim_{t\to\infty}|x(t)|> \gamma(K)\}< \varepsilon$ holds,
then the state $x(t)$ of system (\ref{eq26}) is  ultimately bounded
in probability sense (UB-P).
\end{definition}

Next, one addresses the criteria on noise-to-state stability of random
time-delay systems in the moment sense and in probability sense, respectively.

\begin{theorem}\label{the31}
For system (\ref{eq26}), if there are a function
$V(t,x)\in C^{1,1}([t_0-\tau, \infty) \times \mathbb{R}^n ; \mathbb{R}_+)$,
positive constants $m, c_1, c_2, c_3, c$ and non-negative constant $d_c$
such that
\begin{gather} %equation是公式环境 eqnarray是公式对齐环境%这样这个公式的几行就对齐了
\label{eq31} c_1|x(t)|^m\leq V(t, x)\leq c_2(\sup_{-\tau\leq s\leq 0}|x(t+s)|)^m, \\
%\label{theorem3a2} \frac{\partial V}{\partial t}+\frac{\partial V}{\partial x}f(t, x(t), x(t-\tau(t))
%+c\Big|\frac{\partial V}{\partial x}g(t, x(t), x(t-\tau(t))\Big|^2\leq -c_3(\sup_{-\tau\leq s\leq 0}|x(t+s)|)^m.
%\label{theorem3a2} a_1|x(t)|^m\leq V(t, x(t))\leq a_2(\sup_{-\tau\leq s\leq 0}|x(t+s)|)^m.
\label{eq32} \dot{V}(t,x)\leq -cV(t,x)+c_3|\xi(t)|^r+d_c,
\end{gather}
then (i) system (\ref{eq26}) has a unique solution  on $[t_0-\tau, \infty)$;
(ii) system (\ref{eq26}) is NSS-\textbf{m}-M;
(iii) the state of system (\ref{eq26}) is UB-\textbf{m}-M.
\end{theorem}

\begin{proof}
From (\ref{eq31}), it yields that
\begin{align}\label{eq33}
%V(t,x)\geq 0,V(t-\tau(t),x(t-\tau(t)))\geq 0,\cr
\lim_{|x|\rightarrow\infty}\inf_{t\geq t_0}V(t,x)=\infty.
\end{align}
Taking expectations on both sides of (\ref{eq32}), then
\begin{align}\label{eq34}
E\dot{V}(t,x)\leq -cEV(t, x) +c_3E|\xi(t)|^r +d_c.
\end{align}
According to Lemma \ref{lem2b2}, it follows from (\ref{eq33}) and (\ref{eq34}) that
system (\ref{eq26}) has a unique global solution.

Letting $\nu(t)=EV(t,x)$ and using Fubini's theorem
[Theorem 2.39, \ref{Klebaner2005}], then (\ref{eq34}) can be turned into
$$\nu(t+\varepsilon)\leq \nu(t)+ \int_t^{t+\varepsilon} (-c\nu(s) +c_3K+d_c)ds,
\forall\varepsilon>0, t\geq t_0.$$
So, $$D^{+}\nu(t)\leq -c\nu(t) +c_3K+d_c.$$
By Lemma \ref{lem2b4}, one arrives at
\begin{align*}
\nu(t)&\leq \nu(t_0)e^{-c(t-t_0)}+\int_{t_0}^{t}e^{-c(t-s)}(c_3K+d_c)ds\cr
&\leq \nu(t_0)e^{-c(t-t_0)}+\frac{c_3K+d_c}{c}-\frac{c_3K+d_c}{c}e^{c(t_0-t)}\cr
&\leq \nu(t_0)e^{-c(t-t_0)}+\frac{c_3K+d_c}{c},
\end{align*}
together with (\ref{eq31}),  one can deduce that
\begin{align}\label{eq35}
E|x(t)|^m \leq \frac{c_2}{c_1}|\varphi|^m e^{-c(t-t_0)} + \frac{c_3K+d_c}{cc_1}.
\end{align}
Thus, the system (\ref{eq26}) is NSS-\textbf{m}-M.

In addition, letting $t\rightarrow \infty$ in (\ref{eq35}), it yields that
\begin{align*}
\lim_{t\rightarrow \infty}E|x(t)|^m \leq \frac{c_3K+d_c}{cc_1}.
\end{align*}
This means that the state of system (\ref{eq26}) is UB-\textbf{m}-M.
\end{proof}

\begin{theorem}\label{the32}
For system (\ref{eq26}), if there are a function $V(x)\in
C(\mathbb{R}^n ; \mathbb{R}_+)$, functions $\gamma \in \mathcal{K}$,
$\gamma_1 \in \mathcal{KL}$, $\gamma_2 \in \mathcal{KL}$ and
positive constant $a_0$ such that
\begin{gather} %equation是公式环境 eqnarray是公式对齐环境%这样这个公式的几行就对齐了
\label{eq36} \gamma_1(|x(t)|)\leq V(x)\leq  \gamma_2(\sup_{-\tau\leq s\leq 0}|x(t+s)|), \\
%\label{theorem3a2} \frac{\partial V}{\partial t}+\frac{\partial V}{\partial x}f(t, x(t), x(t-\tau(t))
%+c\Big|\frac{\partial V}{\partial x}g(t, x(t), x(t-\tau(t))\Big|^2\leq -c_3(\sup_{-\tau\leq s\leq 0}|x(t+s)|)^m.
%\label{theorem3a2} a_1|x(t)|^m\leq V(t, x(t))\leq a_2(\sup_{-\tau\leq s\leq 0}|x(t+s)|)^m.
\label{eq37} \dot{V}(x)\leq -\gamma (\sup_{-\tau\leq s\leq 0}|x(t+s)|)+a_0|\xi(t)|^r,
\end{gather}
then system (\ref{eq26}) has a unique solution.
If $\gamma \circ \gamma_2^{-1}(\cdot)$ is a convex function, then  system (\ref{eq26})
is NSS-P and the state of system (\ref{eq26}) is UB-P.
\end{theorem}

\begin{proof}
From (\ref{eq36}) and (\ref{eq37}), one  obtains  (\ref{eq33}) and
\begin{align}\label{eq38}
E\dot{V}(x)&\leq -E[\gamma (\sup_{-\tau\leq s\leq 0}|x(t+s)|)]+a_0E|\xi(t)|^r\cr
&\leq a_0K(1+EV(x)+EV( x(t-\tau(t))),
\end{align}
respectively.
So system (\ref{eq26}) has a unique solution in the light of Lemma \ref{lem2b2}.

If $\gamma \circ \gamma_2^{-1}(\cdot)$ is a convex function, then by
(\ref{eq38}) and Jensen's inequality [\ref{Hardy1989}], and  let
$\bar{\nu}(t)=EV(x)$, one gets
\begin{align*}
\dot{\bar{\nu}}(t)&\leq -E[\gamma (\sup_{-\tau\leq s\leq 0}|x(t+s)|)]+a_0E|\xi(t)|^r\cr
&\leq -\gamma \circ \gamma_2^{-1}(\bar{\nu}(t))+a_0\sup_{ t_0\leq l \leq t} E|\xi(l)|^r.
\end{align*}
Following the thread of Theorem 2 of [\ref{Wu2015}], there exist functions
$\tilde{\beta} \in \mathcal{KL}$ and $\tilde{\gamma}  \in \mathcal{K}$ such that
$P\{|x(t)|>\tilde{\beta}(|\varphi|, t-t_0)+\tilde{\gamma}(a_0\sup_{ t_0\leq l \leq t}
E|\xi(l)|^r)\}< \varepsilon $ holds, and letting $t\rightarrow\infty$, then one has
\begin{align*}
P\{\lim_{t\to\infty}|x(t)|> \tilde{\gamma}(a_0K)\}< \varepsilon.
\end{align*}
Thus, system (\ref{eq26}) is NSS-P and the system state is UB-P.
\end{proof}
\begin{remark}\label{rem31}
(i) In Theorem \ref{the31}, if (\ref{eq32}) is replaced by
$\dot{V}(t,x)\leq -c(\sup_{-\tau\leq s\leq 0}|x(t+s)|)^m+c_3|\xi(t)|^r+d_c$, the result
remains valid. In fact, by (\ref{eq31}), one has
\begin{align*}
\dot{V}(t,x)&\leq -c(\sup_{-\tau\leq s\leq 0}|x(t+s)|)^m+c_3|\xi(t)|^r+d_c\cr
&\leq -\frac{c}{c_2}V(t,x)+c_3|\xi(t)|^r+d_c.
\end{align*}
(ii) From (\ref{eq38}), it is certain that if (\ref{eq37}) is replaced by
$\dot{V}(x)\leq -\gamma \circ \gamma_2^{-1}(V(x))+a_0|\xi(t)|^r$,
Theorem \ref{the32} still holds for system (\ref{eq26}).
\end{remark}

\begin{remark}\label{rem32}
The random disturbance $\xi(t)$ in RDDE (\ref{eq26}) satisfies the assumption
$\sup_{t\geq t_0} E|\xi(t)|^r<K$, under which one gives the criteria
on noise-to-state stability of RDDE (\ref{eq26}).  If
the random disturbance $\xi(t)$ in  RDDE (\ref{eq26}) satisfies that
\begin{align}\label{eq311}
E|\xi(t)|^r\leq \bar{d}_0 e^{\bar{c}_0t}, ~~t\geq t_0,
\end{align}
where $\bar{c}_0$ is a constant and $\bar{d}_0$ is a positive constant, then by
adopting the similar way used in section III of [\ref{Wu2015}],
one knows that the result (i) and (ii)
in Theorem \ref{the31} still hold, but the result (iii) does not necessarily hold.
Similarly, If (\ref{eq36}) and (\ref{eq37}) hold in Theorem \ref{the32} for
RDDE (\ref{eq26}) satisfying (\ref{eq311}), then system (\ref{eq26}) has
a unique solution  on $[t_0-\tau,\infty)$. Furthermore,
system (\ref{eq26}) is also NSS-P if $\gamma \circ \gamma_2^{-1}(\cdot)$ is a convex function,
but the  state of system (\ref{eq26}) is  not necessarily UB-P.
\end{remark}
\section{Application 1: state feedback regulation  control}\label{S5}
%\section{Application 1: adaptive output feedback regulation control}\label{S4}
Consider  the following random nonlinear strict-feedback system with
time-varying delay
\begin{equation}\label{equ41}
\begin{cases}
\dot{x}_1(t)=x_2(t)+f_1+g_1\xi_1(t),\\
\dot{x}_2(t)=x_3(t)+f_2+g_2\xi_2(t),\\
~~~~~~~~\vdots\\
\dot{x}_{n}(t)=u(t)+f_{n}+g_{n}\xi_{n}(t),\\
y(t)=x_1(t).
\end{cases}
\end{equation}
where $f_i=f_i(t, \bar{x}_i(t-\tau(t)), \bar{x}_i(t) )$,
$g_i=g_i(t, \bar{x}_i(t-\tau(t)),  \bar{x}_i(t))$,
$\bar{x}_i(t-\tau(t))=[x_1(t-\tau(t)),\cdots,x_i(t-\tau(t))]^T$
and $\bar{x}_i(t)=[x_1(t),\cdots,x_i(t)]^T$.
System state is $x(t)=[x_1(t), x_2(t), \cdots, x_n(t)]^T\in \mathbb{R}^{n}$
with the initial value
$x_{t_0}=\varphi\in C_{\mathcal{F}_{t_0}}^b([t_0-\tau, t_0]; \mathbb{R}^n)$, and
system output $y(t)\in \mathbb{R}$ and system input $u(t)\in \mathbb{R}$.
Borel measurable function $\tau(t):[t_0,\infty)\rightarrow [0,\tau]$
denotes time-varying delay and satisfies $\dot{\tau}(t)\leq
\tau^*<1$, where $\tau$ and  $\tau^*$ are known constants.
Stochastic process $\xi(t)=[\xi_1(t), \xi_2(t), \cdots, \xi_n(t)]^T \in \mathbb{R}^{n}$
defined on the complete filtered probability space $(\Omega, \mathcal{F},
\{\mathcal {F}_t\}_{t\geq t_0}, P)$ is $\mathcal{F}_t$-adapted and
piecewise continuous. $f_i, g_i, i=1, 2, \cdots, n$ are
locally Lipschitz continuous functions in  $\bar{x}_i(t-\tau(t))$ and $\bar{x}_i(t)$, respectively.
$f(t, 0, 0)$ and $g(t, 0, 0)$ are bounded.

For system (\ref{equ41}), one proposes the following assumptions:
\textbf{H0}: The random disturbance $\xi(t)$ satisfies
$$\sup_{t\geq t_0} E|\xi(t)|^4<K. $$
\textbf{H1}: The system state $x(t)=[x_1(t), \cdots, x_n(t)]^T$
is available.\\
\textbf{H2}: There exists a positive constant $\bar{\theta}$  such that
for every  $i\in\{1, 2, \cdots, n\}$,
\begin{align*}
|f_i(t,\bar{x}_i(t-\tau(t)), \bar{x}_i(t))|&\leq \bar{\theta}\sum_{j=1}^i(|x_j(t-\tau(t))|+|x_j(t)|),\cr
|g_i(t,\bar{x}_i(t-\tau(t)), \bar{x}_i(t))|^2&\leq  \bar{\theta}\sum_{j=1}^i(|x_j(t-\tau(t))|+|x_j(t)|).
\end{align*}

The aim  in this section is to design a  controller by state feedback method such that
the output $y(t)$ can be asymptotically regulated  to a neighborhood of
the zero that is  arbitrarily small and the system state is ultimately
bounded in the mean-square sense.

First, introducing the transformations
\begin{align}\label{equ51}
z_1(t)=y(t), z_{i+1}(t)=x_{i+1}(t)-\alpha_{i}(y(t),\bar{x}_{i-1}(t)),
%i=1, 2, \cdots, n,
\end{align}
where $i=1, 2, \cdots, n$ and $x_{n+1}(t)=u(t)$,
$z_{n+1}(t)=0$, $\bar{x}_{0}(t)=0$, and
stabilizing functions $\alpha_{i}(i=1, \cdots, n)$ are to be
constructed.

Next, one constructs the Lyapunov-Krasovskii functional step by
step.

\textbf{Step 1} Choosing
\begin{equation*}
V_1=V_{\tau 1}+\frac{1}{2}z_1^2, ~~
V_{\tau 1}=n\int_{t-\tau(t)}^te^{s-t}z_1^2(s)ds,
\end{equation*}
and noting that  $\dot{\tau}(t)\leq \tau^*<1$ and $0\leq \tau(t)\leq \tau$,
the derivative of $V_1$ satisfies
\begin{align}\label{equ52}
\dot{V}_1\leq &z_1(\alpha_{1}+z_2+f_1+g_1\xi_1)+nz_1^2\cr
&-ne^{-\tau}(1-\tau^*)z_1^2(t-\tau(t))-V_{\tau 1}.
\end{align}
Applying Young's inequality and (\ref{equ51}), one gets
\begin{align}
\label{equ531} z_1z_2 &\leq \frac{1}{2}z_1^2+\frac{1}{2}z_2^2, \\
\label{equ532}z_1f_1 &\leq \Omega_1z_1^2+\frac{(1-\tau^*)e^{-\tau}}{2}z_1^2(t-\tau(t)),\\
\label{equ533}
z_1g_1\xi_1&\leq \Lambda_{1}z_1^2+\frac{(1-\tau^*)e^{-\tau}}{2}z_1^2(t-\tau(t))+\pi_1|\xi_1|^4,
\end{align}
where
$\Omega_1=\bar{\theta}+\frac{1}{2(1-\tau^*)}\bar{\theta}^2e^{\tau}$,
$\Lambda_{1}=\frac{1}{2}\bar{\theta}+\frac{1}{8(1-\tau^*)}\bar{\theta}^2e^{\tau}
+\frac{1}{16\pi_1}$, and $\pi_1$ is a positive design parameter.
Substituting (\ref{equ531})-(\ref{equ533}) into (\ref{equ52}) yields
\begin{align*}
\dot{V}_1\leq&
z_1(nz_1+\Omega_1z_1+\Lambda_{1}z_1+\frac{1}{2}z_1+\alpha_{1})
+\frac{1}{2}z_2^2+\pi_1|\xi_1|^4\cr
&-(n-1)e^{-\tau}(1-\tau^*)z_1^2(t-\tau(t))-V_{\tau
1}.
\end{align*}
Selecting the stabilizing function
$\alpha_{1}=-2nz_1-\Omega_1z_1-\Lambda_{1}z_1-\frac{1}{2}z_1
\triangleq -\beta_1z_1$, then one  has
\begin{align*}
\dot{V}_1\leq & \frac{1}{2}z_2^2+\pi_1|\xi_1|^4-nz_1^2-V_{\tau 1}\cr
&-(n-1)(1-\tau^*)e^{-\tau}z_1^2(t-\tau(t)).
\end{align*}

\textbf{Step i} $(i=2, 3, \cdots, n)$ Suppose that
$V_{i-1}=V_{i-2}+\frac{1}{2}z_{i-1}^2+(n-i+2)\int_{t-\tau(t)}^te^{s-t}z_{i-1}^2(s)ds$
and a series of stabilizing functions $\alpha_j(j=1, 2, \cdots,i-1)$
have been constructed  such that
\begin{align}\label{equ54}
\dot{V}_{i-1}\leq& \frac{1}{2}z_i^2+\sum_{j=1}^{i-1}\pi_j|\xi_j|^4
+\sum_{j=1}^{i-2}\sum_{l=j+1}^{i-1}\pi_{lj}|\xi_j|^4\cr
&-(n-i+1)e^{-\tau}(1-\tau^*)\sum_{j=1}^{i-1}z_j^2(t-\tau(t))\cr
&-(n-i+2)\sum_{j=1}^{i-1}z_j^2 -\sum_{j=1}^{i-1}V_{\tau j},
\end{align}
where $V_{\tau j}=(n-j+1)\int_{t-\tau(t)}^te^{s-t}z_{j}^2(s)ds, j=1,\cdots, i-1$.
In what follows, one will prove that (\ref{equ54}) also holds
for $V_i$.

Let
$V_{i}=V_{i-1}+\frac{1}{2}z_{i}^2+(n-i+1)\int_{t-\tau(t)}^te^{s-t}z_{i}^2(s)ds$,
then
\begin{align}\label{equ55}
\dot{V}_{i}\leq&
\dot{V}_{i-1}+z_{i}(z_{i+1}+\alpha_{i}+f_i+g_i\xi_i)+(n-i+1)z_i^2\cr
&-z_{i}\sum_{j=1}^{i-1}\frac{\partial\alpha_{i-1}}{\partial x_j}
(x_{j+1}+f_j+g_j\xi_j)-V_{\tau i}\cr
&-(n-i+1)e^{-\tau}(1-\tau^*)z_i^2(t-\tau(t)),
\end{align}
where
\begin{align}\label{equ56}
-\frac{\partial\alpha_{i-1}}{\partial x_j}
=\prod_{k=j}^{i-1}\beta_k\triangleq \Xi_{j(i-1)}, ~~j=1, 2, \cdots
i-1.
\end{align}
By \textbf{H2}, (\ref{equ56}) and Young's inequality, one obtains
\begin{align}
\label{equ570}
&z_iz_{i+1}
\leq \frac{1}{2}z_i^2+\frac{1}{2}z_{i+1}^2,\\
\label{equ571}
&z_if_i-z_i\sum_{j=1}^{i-1}\frac{\partial\alpha_{i-1}}{\partial x_j}(x_{j+1}+f_j)\cr
\leq &\sum_{l=1}^{i-1}\frac{1}{2}z_l^2+\Omega_iz_i^2
+\frac{e^{-\tau}(1-\tau^*)}{2}\sum_{l=1}^{i}z_l^2(t-\tau(t)),\\
\label{equ572}
&z_ig_i\xi_i-z_i\sum_{j=1}^{i-1}\frac{\partial\alpha_{i-1}}{\partial x_j}g_j\xi_j\cr
\leq &\sum_{l=1}^{i-1}\frac{1}{2}z_l^2+\Lambda_iz_i^2
+\frac{e^{-\tau}(1-\tau^*)}{2}\sum_{l=1}^{i}z_l^2(t-\tau(t))\cr
&+\pi_i|\xi_i|^4+\sum_{j=1}^{i-1}\pi_{ij}|\xi_j|^4,
\end{align}
where $\pi_i$ and $\pi_{ij}$ are positive design parameters. The
detail proofs of (\ref{equ571}) and (\ref{equ572}) are given in
\textbf{Appendix A} and \textbf{Appendix B}, respectively.
Substituting (\ref{equ54}), (\ref{equ570})-(\ref{equ572}) to
(\ref{equ55}) leads
\begin{align}\label{equ58}
\dot{V}_{i}\leq &z_{i}(z_{i}+\alpha_{i}+2(n-i+1)z_{i}
+\Omega_iz_i+\Lambda_{i}z_{i})\cr
&+\frac{1}{2}z_{i+1}^2+\sum_{j=1}^{i}\pi_j|\xi_j|^4
+\sum_{j=1}^{i-1}\sum_{l=j+1}^{i}\pi_{lj}|\xi_j|^4\cr
&-(n-i+1)\sum_{j=1}^{i}z_j^2-\sum_{j=1}^{i}V_{\tau j}\cr
&-(n-i)e^{-\tau}(1-\tau^*)\sum_{j=1}^{i}z_j^2(t-\tau(t)).
\end{align}
Choosing the stabilizing function
\begin{align*}
\alpha_{i}=-z_{i}-2(n-i+1)z_{i}-\Omega_iz_i-\Lambda_{i}z_{i}
\triangleq -\beta_iz_i,
\end{align*}
then (\ref{equ58}) can be changed into
\begin{align*}
\dot{V}_{i}\leq&  \frac{1}{2}z_{i+1}^2
+\pi_i|\xi_i|^4+\sum_{j=1}^{i-1}(\pi_j+\sum_{l=j+1}^{i}\pi_{lj})|\xi_j|^4
-\sum_{j=1}^{i}V_{\tau j}\cr
&-(n-i)e^{-\tau}(1-\tau^*)\sum_{j=1}^{i}z_j^2(t-\tau(t))\cr
&-(n-i+1)\sum_{j=1}^{i}z_j^2.
\end{align*}
Let $i=n$, together with $z_{n+1}=0$, then one gets the actual
control law
\begin{align}\label{equ59}
u=-z_{n}-2z_{n}-\Omega_nz_n-\Lambda_{n}z_{n}
%\cr
%&=-\beta_nz_n
= -\sum_{i=1}^{n}(\prod_{j=i}^{n}\beta_j)x_i,
\end{align}
and
\begin{align}\label{equ510}
\dot{V}_{n}
\leq &-\sum_{j=1}^{n}\frac{1}{2}z_j^2-\sum_{j=1}^{n}V_{\tau j}+\pi_n|\xi_n|^4\cr
&+\sum_{j=1}^{n-1}(\pi_j+\sum_{l=j+1}^{n}\pi_{lj})|\xi_j|^4\cr
\leq &-V_{n}+\tilde{\pi}|\xi|^4,
\end{align}
where $\tilde{\pi}=\max\{\max_{1\leq j\leq
n-1}\{\pi_j+\sum_{l=j+1}^{n}\pi_{lj}\}, \pi_n\}$.

According to the above design procedure, for system (\ref{equ41})
satisfying the assumptions \textbf{H0}, \textbf{H1} and
\textbf{H2}, one  chooses the  Lyapunov-Krasovskii
functional
\begin{align*}
V(t,z(t))%&=\sum_{i=1}^n\frac{1}{2}z_{i}^2(t)+\sum_{i=1}^n V_{\tau i}.
&=\sum_{i=1}^n\Big[\frac{1}{2}z_{i}^2(t)+(n-i+1)\int_{t-\tau(t)}^te^{s-t}z_{i}^2(s)ds\Big].
\end{align*}
It is clear that
\begin{equation*}
V(t,z(t))\geq \sum_{j=1}^n\frac{1}{2}z_j^2=\frac{1}{2}|z(t)|^2
\end{equation*}
and
\begin{align*}
V(t,z(t))&\leq n\sum_{j=1}^n \int_{-\tau(t)}^0z_j^2(v+t)dv
+\sum_{j=1}^n\frac{1}{2}z_j^2\cr &\leq
n^2\tau\sum_{j=1}^n(\sup_{-\tau\leq v\leq
0}|z_j(v+t)|)^2+\frac{1}{2}|z(t)|^2\cr &\leq
\frac{1+2n^2\tau}{2}(\sup_{-\tau\leq v\leq 0}|z(v+t)|)^2,
\end{align*}
that is to say,
\begin{equation}\label{equ591}
\frac{1}{2}|z(t)|^2 \leq V(t,z(t)) \leq
\frac{1+2n^2\tau}{2}(\sup_{-\tau\leq v\leq 0}|z(v+t)|)^2.
\end{equation}

From (\ref{equ510}) and (\ref{equ591}),  the Lyapunov-Krasovskii
functional $V$ meets the requirements of Theorem \ref{the31}, so
the closed-loop system (including (\ref{equ41}) and (\ref{equ59}) )
is noise-to-state stable in the mean-square sense and has a unique
solution which is UB-\textbf{2}-M. Moreover, the output regulation
error of closed-loop system satisfies
\begin{align*}
\lim_{t\rightarrow \infty}E|y(t)|^2 \leq \lim_{t\rightarrow
\infty}E|z(t)|^2 \leq 2\tilde{\pi}K.
\end{align*}
According to the definitions of $\tilde{\pi}$, the regulation error
can be decreased  by tuning the positive constants
$\pi_j+\sum_{i=j+1}^{n}\pi_{ij}(1 \leq j \leq n-1)$ and $\pi_n$
small enough (i.e., letting $\pi_j$, $\pi_{ij}(i=j+1, 1 \leq j \leq
n-1)$ and $\pi_n$ arbitrarily small), where $\pi_j(1 \leq j \leq
n-1)$, $\pi_{ij}(i=j+1, 1 \leq j \leq n-1)$ and $\pi_n$ are
independent of each other. In other words, the regulation error can
be made small enough in the mean-square sense by selecting
appropriate design parameters.

From what has been analyzed above, one draws the stability analysis
result as below at present.
\begin{theorem}
If  system (\ref{equ41}) satisfies the assumptions \textbf{H0},
\textbf{H1} and \textbf{H2}, then by selecting appropriate design
parameters, the closed-loop system (\ref{equ41}) and (\ref{equ59})
is NSS-\textbf{2}-M and has a unique solution which is UB-\textbf{2}-M.
Moreover, the regulation error
satisfies
\begin{align*}
\lim_{t\rightarrow \infty}E|y(t)|^2 \leq 2\tilde{\pi}K,
\end{align*}
and the regulation error can be made arbitrarily small in the
mean-square sense by parameter-tuning technique.
\end{theorem}

\section{Application 2: adaptive output feedback regulation control}\label{S4}

In this section, one focuses on the adaptive output feedback
regulation control for system (\ref{equ41}). To this end,
one supposes that system (\ref{equ41}) satisfies the assumption
\textbf{H0} and the following two assumptions:\\
\textbf{H1'}: Only $x_1(t)$ is measurable by system output $y(t)$,
other state variables $x_2(t), \cdots, x_n(t)$ are unavailable.\\
\textbf{H2'}: For $i=1, 2, \cdots, n$, unknown functions $f_i, g_i$
satisfy
\begin{align*}
|f_i(t, \bar{x}_i(t-\tau(t)),\bar{x}_i(t))|\leq \theta_i[\phi_{\tau i}(y(t-\tau(t)))+\phi_i(y(t))], \cr
|g_i(t,\bar{x}_i(t-\tau(t)),\bar{x}_i(t))|\leq \theta_i[\psi_{\tau i}(y(t-\tau(t)))+\psi_i(y(t))],
\end{align*}
where $\phi_i, \phi_{\tau i}, \psi_i$ and $\psi_{\tau i}$ are smooth
non-negative function with $\phi_i(0)=\phi_{\tau
i}(0)=\psi_i(0)=\psi_{\tau i}(0)=0$. $\theta_i>0$ are unknown
parameters.

It follows from the assumption \textbf{H2'} that there exist smooth
functions $\bar{\phi}_i(y), \bar{\phi}_{\tau i}(y(t-\tau(t))),
\bar{\psi}_i(y)$ and $\bar{\psi}_{\tau i}(y(t-\tau(t)))$
satisfying
\begin{align}\label{equ42}
%\begin{cases}
&\phi_i(y)=y\bar{\phi}_i(y), ~~~~\psi_i(y)=y\bar{\psi}_i(y),\cr
&\phi_{\tau i}(y(t-\tau(t)))=y(t-\tau(t))\bar{\phi}_{\tau i}(y(t-\tau(t))), \cr
&\psi_{\tau i}(y(t-\tau(t)))=y(t-\tau(t))\bar{\psi}_{\tau i}(y(t-\tau(t))).
%\end{cases}
\end{align}

The task in this section is to construct a  regulation controller for system (\ref{equ41})
by adaptive output feedback method such that the system output
$y(t)$ can be regulated  to  an neighborhood  of the zero whose
scope can be controlled and the system state is bounded in the
mean-square sense.

\subsection{Observer design}
Since state variables $x_2(t), \cdots, x_n(t)$ are unavailable, one
reconfigures $x_i(t) (i=2, \cdots, n)$ by the following
reduced-order observer:
\begin{equation}\label{equ43}
\begin{cases}
\dot{\hat{x}}_i(t)=\hat{x}_{i+1}(t)+\kappa_{i+1}y(t)-\kappa_i(\hat{x}_1(t)+\kappa_1y(t)),\\
i=1, 2, \dots n-2, \\
\dot{\hat{x}}_{n-1}(t)=u(t)-\kappa_{n-1}(\hat{x}_1(t)+\kappa_1y(t)),
\end{cases}
\end{equation}
where $\kappa_1, \dots, \kappa_{n-1}$ are the observer gain to be
designed.

Let $\theta^*=\max\{1, \theta_1, \cdots, \theta_n \}$, $e=[e_2,
\cdots, e_n]^T$ with
$e_j=(x_j(t)-\hat{x}_{j-1}(t)-\kappa_{j-1}y(t))/\theta^*(j=2, 3,
\cdots, n)$, then the dynamic trajectory of observer
error can be described by
\begin{equation}\label{equ44}
\dot{e}=A_ee +\frac{1}{\theta^*}\Xi_1+\frac{1}{\theta^*}\Xi_2\xi,
\end{equation}
where $A_e=(-\kappa, \bar{I}_{n-1})$,
$\bar{I}_{n-1}=diag\{1, 1, \cdots, 1, 0\}$,
$\kappa=(\kappa_1, \cdots, \kappa_{n-1})^T$,
$\Xi_1=[f_2-\kappa_1f_1, \cdots, f_n-\kappa_{n-1}f_1]^T$,
$\Xi_2=(-g_1\kappa,\Theta)$,
$\Theta=diag\{g_2, \cdots, g_n\}$.

The observer gain $\kappa_1, \dots, \kappa_{n-1}$ should be chosen
to meet the requirement that $A_e$ is Hurwitz. This means that there
is a matrix $P_{n-1}>0$  such that
$A_e^TP_{n-1}+P_{n-1}A_e+bI_{n-1}=0$, where $b$ is a positive design constant.

\subsection{Adaptive regulation controller design}\label{S42}
First, one introduces the error transformations
\begin{equation}\label{equ45}
z_1=y(t), ~~ z_{i+1}=\hat{x}_i-\alpha_{i}(y, \bar{\hat{x}}_{i-1},
\hat{\theta}),
\end{equation}
where $\bar{\hat{x}}_{i-1}(t)=(\hat{x}_1(t), \cdots,
\hat{x}_{i-1}(t))^T$, $i=1, 2, \cdots, n$, $\hat{x}_{n}=u$, $z_{n+1}=0$,
$\bar{\hat{x}}_{0}=0$, $\alpha_{i}(i=1, \cdots, n)$ are stabilizing
functions to be determined and $\hat{\theta}$ is the estimation of
$\theta$ with $\theta=\max\{\theta^{*2},\theta^{*4}\}$. In the light
of (\ref{equ43})-(\ref{equ45}), one can deduce that
\begin{align}
\label{equ451} \dot{z}_1=&e_2\theta^*+\hat{x}_1+\kappa_1y+f_1+g_1\xi_1,\\
\label{equ452}
\dot{z}_i=&z_{i+1}+\alpha_{i}-\frac{\partial\alpha_{i-1}}{\partial\hat{\theta}}\dot{\hat{\theta}}
-\frac{\partial\alpha_{i-1}}{\partial y}(e_2\theta^*+f_1)\cr
&-\frac{\partial\alpha_{i-1}}{\partial y}g_1\xi_1+\bar{\beta}_i,
~~i=2, \cdots, n,
\end{align}
where
$\bar{\beta}_i=\kappa_iy-\sum_{j=1}^{i-2}\frac{\partial\alpha_{i-1}}{\partial
\hat{x}_j}(\hat{x}_{j+1}+\kappa_{j+1}y-\kappa_i(\hat{x}_1+\kappa_1y))
-(\kappa_{i-1}+\frac{\partial\alpha_{i-1}}{\partial
y})(\hat{x}_1+\kappa_1y)$.

Next, one begins to design adaptive regulation controller.

\textbf{Step 1} Selecting the candidate Lyapunov-Krasovskii functional
\begin{align*}
V_1&=V_0+V_\tau+\frac{1}{2}z_1^2,~~
V_0=e^TPe+\frac{1}{2\mu}\tilde{\theta}^2,\cr
V_\tau&=\int_{t-\tau(t)}^t e^{s-t}Q(y(s))ds,
\end{align*}
where $\tilde{\theta}=\theta-\hat{\theta}$, both the adaptive gain
$\mu>0$ and continuous function $Q(y(s))>0$ are  to be designed. By
(\ref{equ451}), $\dot{\tau}(t)\leq \tau^*<1$ and $0\leq \tau(t)\leq \tau$,
the derivative of $V_1$  satisfies that
\begin{align}\label{equ46}
\dot{V}_1=&\dot{V}_0
+z_1(e_2\theta^*+\hat{x}_1+\kappa_1y+f_1+g_1\xi_1)-V_\tau\cr
&+Q(y(t))-e^{-\tau(t)}(1-\dot{\tau}(t))Q(y(t-\tau(t)))\cr
\leq &-b|e|^2+\frac{2}{\theta^*}e^TP\Xi_1+\frac{2}{\theta^*}e^TP\Xi_2\xi
-\frac{1}{\mu}\tilde{\theta}\dot{\hat{\theta}} -V_\tau\cr
&-e^{-\tau}(1-\tau^*)Q(y(t-\tau(t)))+Q(y)+z_1e_2\theta^*\cr
&+z_1(z_2+\alpha_1+\kappa_1y)+z_1f_1+z_1g_1\xi_1.
\end{align}
It follows from assumption \textbf{H2'}, (\ref{equ42}), Lemma
\ref{lem2b5} and Young's inequality  that
\begin{align}
\label{equ471}
z_1e_2\theta^*\leq &d_{11}|e|^2+\frac{1}{4}d_{11}^{-1}z_1^2\theta,\\
\label{equ472}
z_1f_1\leq &\frac{1}{4}d_{12}^{-1}z_1^2\theta + 2d_{12}y^2\bar{\phi}_1^2
+ 2d_{12}\phi_{\tau 1}^2,\\
\label{equ473}
z_1g_1\xi_1\leq &\frac{1}{8}d_{13}^{-1}z_1^4\theta +
d_{13}^{-1}y^4\bar{\psi}_1^4\cr
&+ d_{13}^{-1}\psi_{\tau 1}^4+d_{13}|\xi|^2,\\
\label{equ474}
\frac{2}{\theta^*}e^TP\Xi_1
\leq &\sum_{i=2}^n[(\bar{\phi}_i^2+\kappa_{i-1}^2\bar{\phi}_1^2)y^2
+(\phi_{\tau i}^2+\kappa_{i-1}^2\phi_{\tau 1}^2)]\cr
&\cdot4||P||_F^2d_{01}^{-1}+d_{01}|e|^2, \\
\label{equ475} \frac{2}{\theta^*}e^TP\Xi_2\xi
&\leq \sum_{i=2}^n[(\bar{\psi}_i^4+\kappa_{i-1}^4\bar{\psi}_1^4)y^4
+(\psi_{\tau i}^4+\kappa_{i-1}^4\psi_{\tau 1}^4)]\cr
&\cdot4(n-1)d_{02}^{-1}||P||_F^4+ |e|^2+d_{02}|\xi|^4,
\end{align}
where $d_{01}, d_{02}, d_{11}, d_{12}$ and $d_{13}$ are positive design
parameters. Substituting (\ref{equ471})-(\ref{equ475}) into (\ref{equ46}), one obtains
\begin{align}\label{equ48}
\dot{V}_1\leq &-b_1|e|^2+z_1(z_2+\alpha_1+\kappa_1y+\omega
+\frac{1}{8}\varpi_1\hat{\theta})+\delta_1 \cr
&+(\frac{1}{8}\varpi_1z_1-\frac{1}{\mu}\dot{\hat{\theta}})\tilde{\theta}
+d_{02}|\xi|^4+d_{13}|\xi|^2-V_\tau\cr
&-e^{-\tau}(1-\tau^*)Q(y(t-\tau(t)))+Q(y),
\end{align}
where
$\omega=4||P||_F^2d_{01}^{-1}\sum_{i=2}^n[\bar{\phi}_i^2+\kappa_{i-1}^2\bar{\phi}_1^2]y
+4(n-1)||P||_F^4d_{02}^{-1}\sum_{i=2}^n[\bar{\psi}_i^4+\kappa_{i-1}^4\bar{\psi}_1^4]y^3
+2d_{12}\bar{\phi}_1^2y+ d_{13}^{-1}\bar{\psi}_1^4y^3$,
$\delta_1=\delta_0+2d_{12}\phi_{\tau 1}^2+ d_{13}^{-1}\psi_{\tau
1}^4$, $\delta_0=4||P||_F^2d_{01}^{-1}\sum_{i=2}^n[\phi_{\tau
i}^2+\kappa_{i-1}^2\phi_{\tau 1}^2]
+4(n-1)||P||_F^4d_{02}^{-1}\sum_{i=2}^n[\psi_{\tau
i}^4+\kappa_{i-1}^4\psi_{\tau 1}^4]$,
$\varpi_1=2d_{11}^{-1}z_1+2d_{12}^{-1}z_1+d_{13}^{-1}z_1^3$,
$b_1=b_0-d_{11}>0$, $b_0=b-d_{01}-1>0$.

Selecting the following tuning function $\lambda_1$ and the
stabilizing function $\alpha_1$ respectively
\begin{align*}
\lambda_1&=\frac{1}{8}\varpi_1 z_1-\frac{1}{\mu}\hat{\theta},\cr
\alpha_1(y, \hat{\theta})&=-c_1z_1-\kappa_1y-\omega
-\frac{1}{8}\varpi_1\hat{\theta}-\sigma(y),
\end{align*}
where positive constant $c_1$ and the function $\sigma(y)$ are  to
be designed, then
\begin{align}\label{equ49}
\dot{V}_1\leq &-b_1|e|^2+z_1z_2-c_1z_1^2-\sigma(y)y
+(\lambda_1-\frac{1}{\mu}\dot{\hat{\theta}})\tilde{\theta}\cr
&+\frac{1}{\mu}\hat{\theta}\tilde{\theta}+\delta_1
+d_{02}|\xi|^4+d_{13}|\xi|^2-V_\tau\cr
&-e^{-\tau}(1-\tau^*)Q(y(t-\tau(t)))+Q(y).
\end{align}

\textbf{Step k} $(k= 2, \cdots, n)$ Suppose that the appropriate
stabilizing functions $\alpha_j(j=1, 2, \cdots, k-1)$ and tuning
functions $\lambda_j(j=1, 2, \cdots, k-1)$ have been designed such
that $V_{k-1}=V_{k-2}+\frac{1}{2}z_{k-1}^2$ satisfies
\begin{align}\label{equ410}
\dot{V}_{k-1}\leq
&-b_{k-1}|e|^2-\sum_{j=1}^{k-1}c_jz_j^2+z_{k-1}z_{k}-\sigma(y)y\cr
&+d_{02}|\xi|^4+\sum_{j=1}^{k-1}d_{j3}|\xi|^2
+\frac{1}{\mu}\hat{\theta}\tilde{\theta} +\delta_{k-1}\cr
&+\sum_{j=2}^{k-1}[2d_{j2}\bar{\phi}_1^2+d_{13}^{-1}\bar{\psi}_1^4
y^2]y^2
+(\lambda_{k-1}-\frac{1}{\mu}\dot{\hat{\theta}})\tilde{\theta}\cr
&+\mu\sum_{j=2}^{k-1}z_j\frac{\partial\alpha_{j-1}}{\partial\hat{\theta}}
(\lambda_{k-1}-\frac{1}{\mu}\dot{\hat{\theta}})-V_\tau\cr
&-e^{-\tau}(1-\tau^*)Q(y(t-\tau(t)))+Q(y),
\end{align}
where $\delta_{k-1}=\delta_{k-2}+2d_{(k-1)2}\phi_{\tau 1}^2+
d_{(k-1)3}^{-1}\psi_{\tau 1}^4$, $b_{k-1}=b_{k-2}-d_{(k-1)1}>0$.

In what follows, one selects $V_k=V_{k-1}+\frac{1}{2}z_{k}^2$ in
\textbf{Step k} to prove that $V_k$ also satisfies (\ref{equ410}).
From  (\ref{equ452}) and  (\ref{equ410}), one has
\begin{align}\label{equ411}
\dot{V}_{k}\leq &-b_{k-1}|e|^2-\sum_{j=1}^{k-1}c_jz_j^2+\delta_{k-1}+d_{02}|\xi|^4\cr
&+\sum_{j=1}^{k-1}d_{j3}|\xi|^2+z_{k}(z_{k-1}+z_{k+1}+\alpha_{k}+\bar{\beta}_k)\cr
&-z_{k}\frac{\partial\alpha_{k-1}}{\partial y}(e_2\theta^*+f_1+g_1\xi_1)
-z_k\frac{\partial\alpha_{k-1}}{\partial\hat{\theta}}\dot{\hat{\theta}}\cr
%&+\delta_{k-1}+d_{02}|\xi|^4+\sum_{j=1}^{k-1}d_{j3}|\xi|^2\cr
&+\sum_{j=2}^{k-1}[2d_{j2}\bar{\phi}_1^2+d_{j3}^{-1}\bar{\psi}_1^4y^2]y^2
+\frac{1}{\mu}\hat{\theta}\tilde{\theta}-\sigma(y)y\cr
&+(\lambda_{k-1}-\frac{1}{\mu}\dot{\hat{\theta}})\tilde{\theta}
+\mu\sum_{j=2}^{k-1}z_j\frac{\partial\alpha_{j-1}}{\partial\hat{\theta}}(\lambda_{k-1}
-\frac{1}{\mu}\dot{\hat{\theta}}) \cr
&-V_\tau-e^{-\tau}(1-\tau^*)Q(y(t-\tau(t)))+Q(y).
\end{align}

By Young's inequality, assumption \textbf{H2'} and (\ref{equ42}), one
gets the following inequalities
\begin{align}
\label{equ4121} -z_k\frac{\partial\alpha_{k-1}}{\partial y}e_2\theta^*\leq & d_{k1}|e|^2
+\frac{1}{4}d_{k1}^{-1}z_k^2\Big(\frac{\partial\alpha_{k-1}}{\partial y}\Big)^2\theta,\\
\label{equ4122} -z_k\frac{\partial\alpha_{k-1}}{\partial y}f_1
\leq &\frac{1}{4}d_{k2}^{-1}z_k^2\Big(\frac{\partial\alpha_{k-1}}{\partial y}\Big)^2\theta
+ 2d_{k2}y^2\bar{\phi}_1^2\cr
&+ 2d_{k2}\phi_{\tau 1}^2,\\
\label{equ4123} -z_k\frac{\partial\alpha_{k-1}}{\partial y}g_1\xi_1
\leq &
\frac{1}{8}d_{k3}^{-1}z_k^4\Big(\frac{\partial\alpha_{k-1}}{\partial
y}\Big)^4\theta + d_{k3}^{-1}y^4\bar{\psi}_1^4\cr
&+d_{k3}^{-1}\psi_{\tau 1}^4+d_{k3}|\xi|^2.
\end{align}
Substituting (\ref{equ4121})-(\ref{equ4123}) into (\ref{equ411}),
then $\dot{V}_k$ satisfies
\begin{align}\label{equ413}
\dot{V}_{k}\leq &-b_{k}|e|^2-\sum_{j=1}^{k-1}c_jz_j^2+(\frac{1}{8}z_k \varpi_k +\lambda_{k-1}
-\frac{1}{\mu}\dot{\hat{\theta}})\tilde{\theta}\cr
&+z_{k}(z_{k-1}+z_{k+1}+\alpha_{k}+\bar{\beta}_k
+\frac{1}{8}\varpi_k\hat{\theta}
-\mu\frac{\partial\alpha_{k-1}}{\partial\hat{\theta}}\lambda_k)\cr
&-\sigma(y)y+\delta_{k}+\mu z_k\frac{\partial\alpha_{k-1}}{\partial\hat{\theta}}
(\lambda_k-\frac{1}{\mu}\dot{\hat{\theta}})+\frac{1}{\mu}\hat{\theta}\tilde{\theta}\cr
&+\mu\sum_{j=2}^{k-1}z_j\frac{\partial\alpha_{j-1}}{\partial\hat{\theta}}
(\lambda_{k-1}-\frac{1}{\mu}\dot{\hat{\theta}})+d_{02}|\xi|^4\cr
&+\sum_{j=1}^{k}d_{j3}|\xi|^2
+\sum_{j=2}^{k}[2d_{j2}\bar{\phi}_1^2+d_{j3}^{-1}\bar{\psi}_1^4y^2]y^2 \cr
&- e^{-\tau}(1-\tau^*)Q(y(t-\tau(t)))+Q(y)-V_\tau,
\end{align}
where $d_{k1}, d_{k2}$ and $d_{k3}$ are positive design constants.
$b_{k}=b_{k-1}-d_{k1}>0$,
$\varpi_k=2d_{k1}^{-1}z_k\Big(\frac{\partial\alpha_{k-1}}{\partial
y}\Big)^2 +2d_{k2}^{-1}z_k\Big(\frac{\partial\alpha_{k-1}}{\partial
y}\Big)^2 +d_{k3}^{-1}z_k^3\Big(\frac{\partial\alpha_{k-1}}{\partial
y}\Big)^4$ and $\delta_{k}=\delta_{k-1}+2d_{k2}\phi_{\tau 1}^2+
d_{k3}^{-1}\psi_{\tau 1}^4$.

First, let the tuning function
$\lambda_k=\lambda_{k-1}+\frac{1}{8}z_k\varpi_k$, then
\begin{align}
\label{equ4141}
&\frac{1}{8}z_k \varpi_k \tilde{\theta}
+(\lambda_{k-1}-\frac{1}{\mu}\dot{\hat{\theta}})\tilde{\theta}
=(\lambda_{k}-\frac{1}{\mu}\dot{\hat{\theta}})\tilde{\theta},\\
\label{equ4142}
&\mu\sum_{j=2}^{k-1}z_j\frac{\partial\alpha_{j-1}}{\partial\hat{\theta}}
(\lambda_{k}-\frac{1}{\mu}\dot{\hat{\theta}})
-\frac{1}{8}z_k\varpi_k\mu\sum_{j=2}^{k-1}z_j\frac{\partial\alpha_{j-1}}{\partial\hat{\theta}}\cr
=&\mu\sum_{j=2}^{k-1}z_j\frac{\partial\alpha_{j-1}}{\partial\hat{\theta}}
(\lambda_{k-1}-\frac{1}{\mu}\dot{\hat{\theta}}).
\end{align}
Next, one substitutes (\ref{equ4141})-(\ref{equ4142}) into
(\ref{equ413}) and chooses the stabilizing function
\begin{align}
\alpha_k(y, \bar{\hat{x}}_{k-1}, \hat{\theta})
=&-c_kz_k-z_{k-1}-\bar{\beta}_k
+\mu\frac{\partial\alpha_{k-1}}{\partial\hat{\theta}}\lambda_k\cr
&-\frac{1}{8}\varpi_k\hat{\theta}+\frac{1}{8}\varpi_k\mu\sum_{j=2}^{k-1}
z_j\frac{\partial\alpha_{j-1}}{\partial\hat{\theta}},
\end{align}
then (\ref{equ413}) can be changed into
\begin{align*}
\dot{V}_{k}\leq &-b_{k}|e|^2 -\sum_{j=1}^{k}c_jz_j^2+z_{k}z_{k+1}
+(\lambda_{k}-\frac{1}{\mu}\dot{\hat{\theta}})\tilde{\theta}\cr
&+\mu\sum_{j=2}^{k}z_j\frac{\partial\alpha_{j-1}}{\partial\hat{\theta}}
(\lambda_{k}-\frac{1}{\mu}\dot{\hat{\theta}})
+\frac{1}{\mu}\hat{\theta}\tilde{\theta}
+d_{02}|\xi|^4+\delta_{k}\cr &+\sum_{j=1}^{k}d_{j3}|\xi|^2
-\sigma(y)y+\sum_{j=2}^{k}[2d_{j2}\bar{\phi}_1^2
+d_{j3}^{-1}\bar{\psi}_1^4 y^2]y^2 \cr
&-e^{-\tau}(1-\tau^*)Q(y(t-\tau(t)))+Q(y)-V_\tau.
\end{align*}

At \textbf{Step n}, let $Q(y)=\bar{Q}(y)y$,
$\sigma(y)=\bar{Q}(y)+\sum_{j=2}^{n}[2d_{j2}\bar{\phi}_1^2(y)
+d_{j3}^{-1}\bar{\psi}_1^4(y) y^2]y$, where
\begin{align*}
\bar{Q}(y(t))
=&\frac{4e^\tau }{1-\tau^*}||P||_F^2d_{01}^{-1}y\sum_{i=2}^n
[\bar{\phi}_{\tau i}^2(y)+\kappa_{i-1}^2\bar{\phi}_{\tau 1}^2(y)]\cr
&+\frac{e^\tau }{1-\tau^*}\sum_{j=1}^{n}[2d_{j2}\bar{\phi}_{\tau1}^2(y)
+d_{j3}^{-1}\bar{\psi}_{\tau 1}^4(y)y^2]y\cr
&+\frac{4(n-1)e^\tau }{1-\tau^*}||P||_F^4d_{02}^{-1}y^3\cr
&\cdot\sum_{i=2}^n[\bar{\psi}_{\tau i}^4(y)+\kappa_{i-1}^4\bar{\psi}_{\tau 1}^4(y)].
\end{align*}
Furthermore, selecting the update law
\begin{align}\label{equ415}
\dot{\hat{\theta}}=\mu\lambda_{n},
~~\lambda_n=\lambda_{n-1}+\frac{1}{8}z_n\varpi_n,
\end{align}
and the control law
\begin{align}\label{equ416}
u=&-c_nz_n-z_{n-1}-\bar{\beta}_n-\frac{1}{8}\varpi_n\hat{\theta}
+\mu\frac{\partial\alpha_{n-1}}{\partial\hat{\theta}}\lambda_n\cr
&+\frac{1}{8}\varpi_n\mu\sum_{j=2}^{n-1}
z_j\frac{\partial\alpha_{j-1}}{\partial\hat{\theta}},
\end{align}
then one has
\begin{align}\label{equ417}
\dot{V}_{n} \leq &-b_{n}|e|^2-\sum_{j=1}^{n}c_jz_j^2
+\frac{1}{\mu}\hat{\theta}\tilde{\theta}+d_{02}|\xi|^4\cr
&+\sum_{j=1}^{n}d_{j3}|\xi|^2-V_\tau\cr
\leq &-\frac{b_n}{\sigma_M}e^TPe-\frac{1}{2\mu}\tilde{\theta}^2-V_\tau
-\sum_{j=1}^{n}c_jz_j^2 +\frac{1}{2\mu}\theta^2\cr
&+d_{02}|\xi|^4+\sum_{j=1}^{n}d_{j3}(\frac{1}{4}+|\xi|^4)\cr \leq
&-\bar{c}V_n+\bar{d}_1|\xi|^4+\bar{d}_2,
\end{align}
where $\sigma_M=\lambda_{max}(P)$,
$\bar{c}=\min\{\frac{b_n}{\sigma_M}, 1, 2c_1, 2c_2, \cdots, 2c_n
\}$, $\bar{d}_1=d_{02}+\sum_{j=1}^{n}d_{j3}$, and
$\bar{d}_2=\frac{1}{2\mu}\theta^2+\frac{1}{4}\sum_{j=1}^{n}d_{j3}$.

\subsection{Stability analysis}
\begin{theorem}\label{the41}
Suppose that assumptions \textbf{H0}, \textbf{H1'} and \textbf{H2'}
hold for system (\ref{equ41}), then  the closed-loop system
including (\ref{equ41}), (\ref{equ415}) and (\ref{equ416})
is NSS-\textbf{2}-M and the closed-loop system solution on $[t_0-\tau, \infty)$
is UB-\textbf{2}-M. Furthermore, the regulation error in the
mean-square sense satisfies
\begin{align*}
\lim_{t\rightarrow \infty}E|y(t)|^2 \leq
\frac{2\bar{d}_1K+2\bar{d}_2}{\bar{c}},
\end{align*}
and by selecting appropriate design parameters, the upper bound of the regulation error
can be made arbitrarily small.
\end{theorem}

\begin{proof}
For system (\ref{equ41}), one chooses the Lyapunov-Krasovskii
functional
\begin{equation*}
V=e^TPe+\frac{1}{2\mu}\tilde{\theta}^2+\sum_{j=1}^{n}\frac{1}{2}z_j^2
+\int_{t-\tau(t)}^te^{s-t}Q(y(s))ds ,
\end{equation*}
from which one knows that $V$ is a radially unbounded functional and
\begin{align}\label{equ418}
\dot{V}\leq &-\bar{c}V+\bar{d}_1|\xi|^4+\bar{d}_2,
\end{align}
where $\bar{c}=\min\{\frac{b_n}{\sigma_M}, 1, 2c_1, 2c_2, \cdots,
2c_n \}$, $\bar{d}_1=d_{02}+\sum_{j=1}^{n}d_{j3}$, and
$\bar{d}_2=\frac{1}{2\mu}\theta^2+\frac{1}{4}\sum_{j=1}^{n}d_{j3}$
are all positive constants. Taking expectations on both sides of
(\ref{equ418}), one obtains
\begin{align*}
E\dot{V}&\leq -\bar{c}EV +\bar{d}_1E|\xi(t)|^4+\bar{d}_2.
%&\leq(\bar{d}_1K+\bar{d}_2)(1+EV(t,x)+EV(t-\tau(t), x(t-\tau(t))).
\end{align*}
By Lemma \ref{lem2b2}, the closed-loop system including
(\ref{equ41}), (\ref{equ415}) and (\ref{equ416}) has a unique
solution.

Defining $\tilde{\nu}(t)=EV$ and $z(t)=[z_1(t),  \cdots, z_n(t)]^T$,
one adopts the same way  for (\ref{equ418}) as one has
done for (\ref{eq32}) in the proof of Theorem \ref{the31}, then the
inequality
%\begin{align*}
$\tilde{\nu}(t)\leq
\tilde{\nu}(t_0)e^{-\bar{c}(t-t_0)}+(\bar{d}_1K+\bar{d}_2)/\bar{c}$
holds,
%\end{align*}
together with $V \geq \frac{1}{2}|z(t)|^2$, one gets
\begin{align}\label{equ419}
E|z(t)|^2 \leq 2E(V|_{t=t_0})e^{-\bar{c}(t-t_0)} +
\frac{2\bar{d}_1K+2\bar{d}_2}{\bar{c}},
\end{align}
which implies that the closed-loop system including (\ref{equ41}),
(\ref{equ415}) and (\ref{equ416}) is NSS-\textbf{2}-M. Furthermore,
letting $t\rightarrow \infty$ in (\ref{equ419}) yields
\begin{align}\label{equ420}
\lim_{t\rightarrow \infty}E|z(t)|^2 \leq
\frac{2\bar{d}_1K+2\bar{d}_2}{\bar{c}}.
\end{align}
This means that the state of the closed-loop system including
(\ref{equ41}), (\ref{equ415}) and (\ref{equ416}) is UB-\textbf{2}-M.

Additionally, it follows from (\ref{equ420}) that
\begin{align*}
\lim_{t\rightarrow \infty}E|y(t)|^2 \leq \lim_{t\rightarrow
\infty}E|z(t)|^2 \leq \frac{2\bar{d}_1K+2\bar{d}_2}{\bar{c}},
\end{align*}
which implies that the upper bound of regulation error in the
mean-square sense can be made small enough by  selecting $d_{02},
d_{j3}(1 \leq j \leq n)$ sufficiently small and $\mu$ large enough.
The parameters consisting of $\bar{c}, d_{02}, \mu, d_{j3}(1 \leq j
\leq n)$ are independent from each other. One completes the proof.
\end{proof}

\begin{remark}\label{rem51}
From the above design procedure of adaptive output feedback regulation controller in
this section, the observer is introduced to estimate the unmeasurable states,
which lead to the assumption  \textbf{H0} is proposed.
If the state of system (\ref{equ41}) is measurable
(i.e.,the assumption \textbf{H1} holds for  system (\ref{equ41})), and
system (\ref{equ41}) satisfies the assumptions \textbf{H0'}
(the random disturbance $\xi(t)$ satisfies $\sup_{t\geq t_0} E|\xi(t)|^2<K $)
and \textbf{H2'}, then one can design an adaptive state feedback controller to
achieve regulation control objective. The design procedure of state feedback controller
is as similar as the procedure given above, thus one omits the detailed process.
\end{remark}
\section{Simulation examples}\label{S6}

\textbf{Example 1:} The time-delay phenomenon widespreadly exists in
the chemical industry, such as  the chemical reactor recycle system
([\ref{Hua2009}], [\ref{Chen2011}]).
Consider a  two-stage chemical reactor system that is  modeled as
\begin{equation}\label{equ600}
\begin{cases}
\dot{x}_1=-l_1(t)x_1-\frac{1}{\mu_1}x_1+\frac{1-r_1}{v_1}x_2,\\
\dot{x}_2=\frac{\lambda_2}{v_2}u-l_2(t)x_2-\frac{1}{\mu_2}x_2
+\frac{r_1}{v_2}x_{1t}+\frac{r_2}{v_2}x_{2t},\\
y=x_1,
\end{cases}
\end{equation}
where $x_{1t}=x_1(t-\tau(t))$ and $x_{2t}=x_2(t-\tau(t))$,
$x_1$ and $x_2$ stand for produced streams from the two reactors;
the residence time of two reactors are $\mu_1$ and $\mu_2$; $v_1$ and $v_2$
represent reactor volumes;  $r_1$ and $r_2$ denote the recycle flow rates;
the feed rate is $\lambda_2$;  $\tau(t)=0.1(1-\sin t)$ represents the time delay;
$l_1(t)$ and $l_2(t)$ indicate the uncertain reaction functions.

For system (\ref{equ600}), [\ref{Chen2011}] studies the impact of
the white noise random disturbance for the  reactors by introducing It$\hat{o}$
stochastic differential equation. Here, one considers random differential equation by
introducing  random noise disturbance that is not white noise. Let the uncertain
reaction functions $l_i(t)=l_{i0}+\xi_i(t)/\sqrt{|x_i(t)|}$, $i=1,2$, where $l_{i0}$
are reaction constants and random disturbance $\xi(t)=[\xi_1(t), \xi_2(t)]^T$ satisfies
$\sup_{t\geq 0} E|\xi(t)|^4<K$, and one chooses $r_1=r_2=0.5$, $\mu_1=\mu_2=2$,
$l_{10}=l_{20}=0.5$, $\lambda_2=v_1=v_2=0.5$. System (\ref{equ600}) can be
rewritten as
\begin{equation}\label{equ601}
\begin{cases}
\dot{x}_1=x_2dt+f_1dt+g_1\xi_1,\\
\dot{x}_2=udt+f_2dt+g_2\xi_2,\\
y=x_1,
\end{cases}
\end{equation}
where $f_1=-x_1, g_1=-\sqrt{|x_1|},  f_2=-x_2+x_2(t-\tau(t))+x_1(t-\tau(t)), g_2=-\sqrt{|x_2|}$.
One supposes that the state variables $x_1$ and $x_2$ are available. The random disturbance
 $\xi_1(t)$ and $\xi_2(t)$ are produced by
\begin{align}\label{equ602}
\xi_1(t)&=a_1\cos(\ell_1t+U_1)\sin(\bar{\ell}_1t), \cr
\xi_2(t)&=a_2\cos(\ell_2t+U_2)\sin^2(\bar{\ell}_2t),
\end{align}
where $U_1$ and $U_2$ are random variables uniformly distributed on $[0,2\pi]$
and $a_1, a_2, \ell_1, \ell_2, \bar{\ell}_1, \bar{\ell}_2$ are real constants, then
$E\xi_1^4(t)\leq a_1^4$ and $E\xi_2^4(t)\leq a_2^4$.
It is clear that (\ref{equ601}) satisfies  assumptions \textbf{H0},
\textbf{H1} and \textbf{H2}.

In simulation, one chooses
$a_1=1, a_2=1, \ell_1=1.2, \ell_2=0.8,  \bar{\ell}_1=1.5, \bar{\ell}_2=1$,
$\pi_1=\pi_2=\pi_{21}=1$, and initial value $x_1(0)=0.5, x_2(0)=-0.5$.
The responses of closed-looped system (\ref{equ601}) and (\ref{equ59})
are shown in Figure 1, which illustrates that
the output regulation error of system (\ref{equ601}) can be controlled in
a small enough neighborhood of zero and the another state $x_2(t)$
is bounded in the mean-square sense.
\begin{figure}[htb]
 \centering
\rotatebox{360}{\scalebox{0.3}[0.3]{\includegraphics{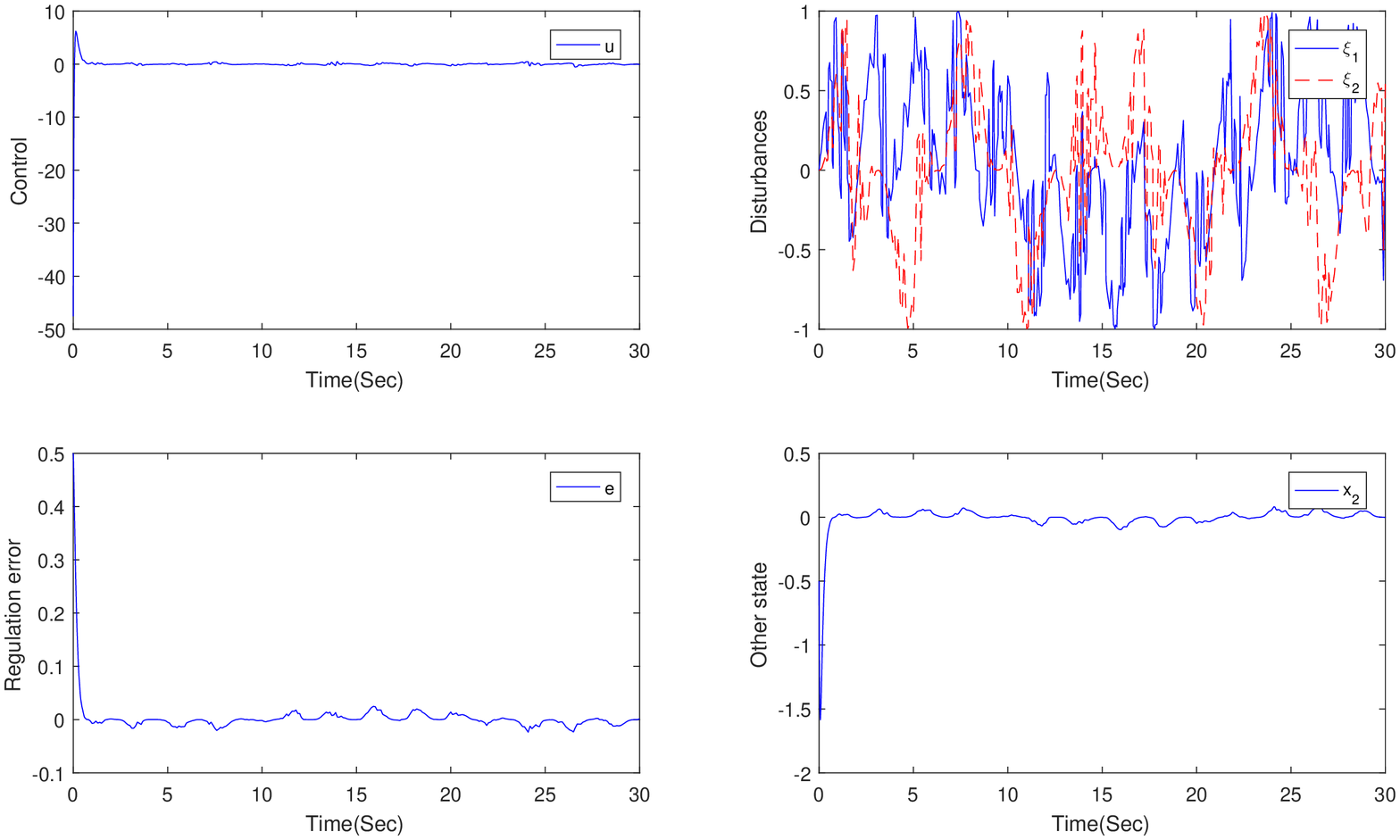}}}
\center{ {\small Figure 1.  The responses of closed-loop system
 (\ref{equ601}) and (\ref{equ59}).}}
\vspace*{0.5cm}
\end{figure}

\textbf{Example 2:} For the following random nonlinear feedback time-delay systems
\begin{equation}\label{equ621}
\begin{cases}
\dot{x}_1=x_2+\theta_1 x_{1t}^2+\theta_1 x_1\xi_1,\\
\dot{x}_2=u+\theta_2 x_1\cos(x_{2t})
+\theta_2 x_{1t}\sin(x_2)\xi_2,\\
y=x_1,
\end{cases}
\end{equation}
where  $x_{1t}=x_1(t-\tau(t))$ and $x_{2t}=x_2(t-\tau(t))$,
$\tau(t)=0.5(1+\cos t)$ that satisfies $0 \leq \tau(t)\leq 1 $
and $\dot{\tau}(t)\leq 0.5$.
Both $\theta_1$ and  $\theta_2$ are unknown parameters.
Only the state variable $x_1$ is measurable
by output $y$, the state variable  $x_2$ is unavailable.
The random disturbance $\xi_1$ and $\xi_2$ are
generated by (\ref{equ602}) with
$a_1=0.8, a_2=1.2, \ell_1=1 , \ell_2=0.5, \bar{\ell}_1=0.5, \bar{\ell}_2=1$.
So system (\ref{equ621})  satisfies assumptions \textbf{H0}, \textbf{H1'} and \textbf{H2'}.

For simulation purpose, one selects
$\theta_1=1.2$ and $\theta_2=1.5$, and initial values
$x_1(0)=0.01$,  $x_2(0)=-0.3$, $\hat{x}_1(0)=0.1, \hat{\theta}(0)=-0.5$,
the design parameters $\mu=1, \kappa_1=1, b=2$,
$d_{01}=1, d_{02}=0.1, d_{11}=d_{12}=1$, $d_{13}=0.1$, $d_{21}=d_{22}=1$,
$d_{23}=0.1$, $c_1=1, c_2=1$. The responses of closed-loop system
(\ref{equ621}), (\ref{equ415}) and (\ref{equ416}) is given in Figure 2, which
shows that the state $x_2$ is bounded in the mean-square sense and the output
regulation error $e$ can be made small enough in the mean-square sense.

\begin{figure}[htb]
 \centering
\rotatebox{360}{\scalebox{0.3}[0.3]{\includegraphics{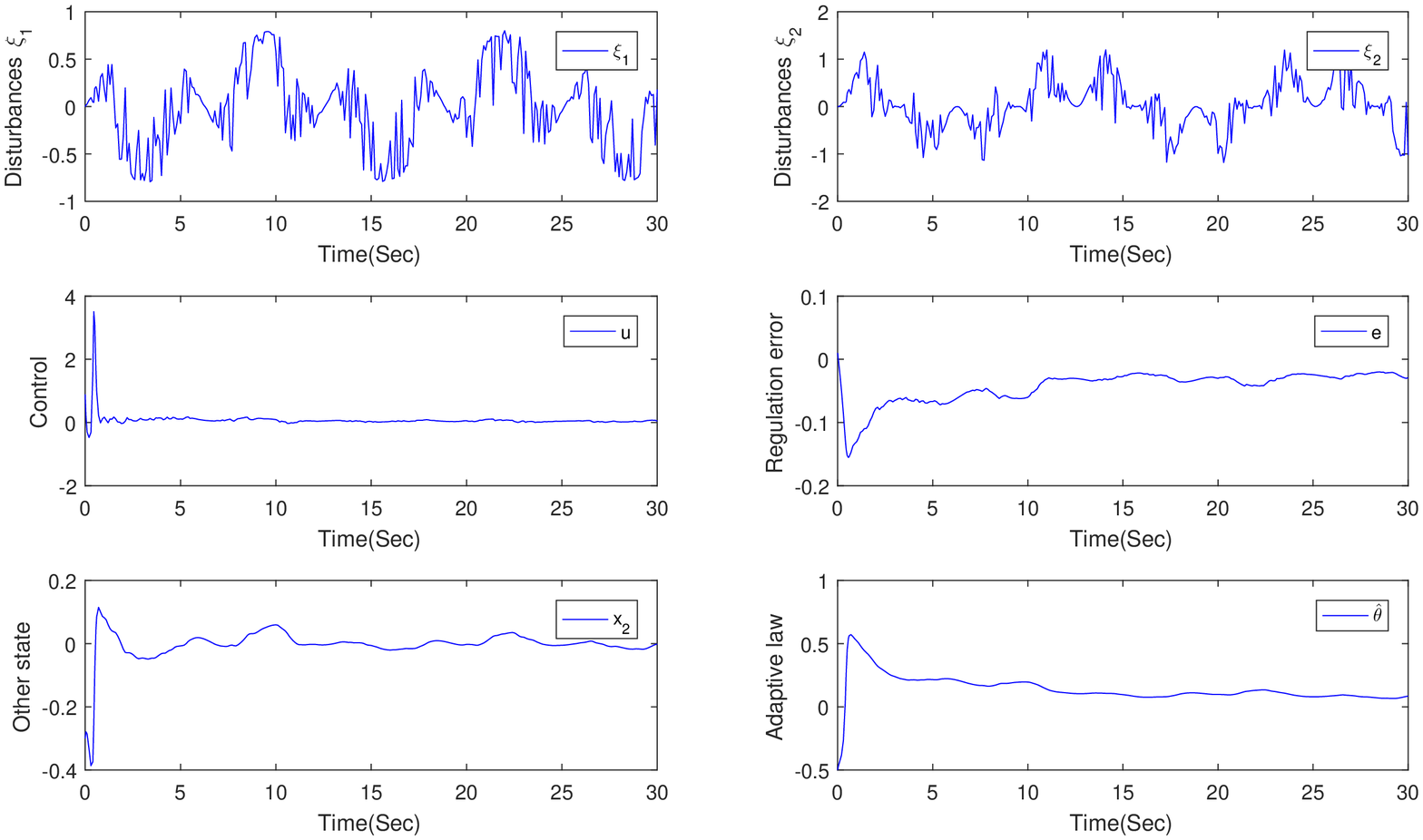}}}
\center{ {\small Figure 2.  The responses of closed-loop system
(\ref{equ621}), (\ref{equ415}) and (\ref{equ416}).}}
\vspace*{0.5cm}
\end{figure}

\section{Conclusion}\label{S7}
This paper investigates the stability of random nonlinear systems with time-delay
and its application in regulation control. The existence and
uniqueness of solution to random nonlinear time-delay systems is guaranteed
by proposed conditions. Next, some criteria on noise-to-state stability of random
nonlinear systems with time-delay are established. As applications,
two kinds of regulation  control problems are solved by state feedback method
and adaptive output feedback method.

Many related problems deserves further investigation. For instance, how
to apply the obtained results in stability analysis and synthesis of
practical systems with time-delay, such as Hamiltonian mechanical systems
and Lagrangian mechanical system; how to generalize this result to random
switched systems with time-delay.

% or


\appendices
%\section*{Appendix A:  }
\section{The Proof of (\ref{equ571})}
\renewcommand{\theequation}{A.\arabic{equation}}\setcounter{equation}{0}
\begin{proof}
From \textbf{H2} and (\ref{equ51}), one can obtain the following inequality
\begin{align}\label{equA1}
&z_if_i-z_i\sum_{j=1}^{i-1}\frac{\partial\alpha_{i-1}}{\partial x_j}(x_{j+1}+f_j)\cr
\leq &|z_i|(\bar{\theta}\sum_{l=1}^i(|x_l|+|x_l(t-\tau(t))|)\cr
&+\sum_{j=1}^{i-1}\Xi_{j(i-1)}
(|x_{j+1}|+\bar{\theta}\sum_{l=1}^j(|x_l|+|x_l(t-\tau(t))|))\cr
\leq & |z_i|\eta_i(\sum_{l=1}^i(|x_l|+|x_l(t-\tau(t))|)
+\sum_{j=1}^{i-1}|x_{j+1}|\cr
&+\sum_{j=1}^{i-1}\sum_{l=1}^j(|x_l|+|x_l(t-\tau(t))|))\cr
=& |z_i|\eta_i(\sum_{l=2}^i(i+2-l)|x_l|\cr
&+\sum_{l=1}^{i}(i+1-l)|x_l(t-\tau(t))|+i|x_1|)\cr
\leq &|z_i|\bar{\eta}_i(|z_1|+\sum_{l=2}^i(|z_l|+\beta_{l-1}|z_{l-1}|)+|z_1(t-\tau(t))|\cr
&+\sum_{l=2}^i(|z_l(t-\tau(t))|+\beta_{l-1}|z_{l-1}(t-\tau(t))|))\cr
\leq &|z_i|\tilde{\eta}_i(\sum_{l=1}^{i-1}(|z_l|+|z_l(t-\tau(t))|)\cr
&+|z_i|\bar{\eta}_{i}(|z_i|+|z_i(t-\tau(t))|),
\end{align}
where
$\eta_i=\max_{1\leq j\leq i-1}\{\bar{\theta}, \Xi_{j(i-1)},\Xi_{j(i-1)}\bar{\theta}\}$,
$\bar{\eta}_{i}=\eta_ii$,
$\tilde{\eta}_{i}=\bar{\eta}_{i}\max_{1\leq l\leq i-1}\{1+\beta_l\}$,
$i=2, 3, \cdots, n$.

By Young's inequality, one gets
\begin{align}
\label{equA2}
|z_i|\tilde{\eta}_i\sum_{l=1}^{i-1}|z_l|\leq& \frac{(i-1)}{2}\tilde{\eta}_i^2 z_i^2
+\sum_{l=1}^{i-1}\frac{1}{2}z_l^2,\\
\label{equA3}
|z_i|\bar{\eta}_{i}|z_i(t-\tau(t))|
\leq&\frac{e^{-\tau}(1-\tau^*)}{2}z_i^2(t-\tau(t))\cr
&+\frac{e^{\tau}}{2(1-\tau^*)}\bar{\eta}_i^2 z_i^2, \\
\label{equA4}
|z_i|\tilde{\eta}_i\sum_{l=1}^{i-1}|z_l(t-\tau(t))|
\leq &\frac{e^{-\tau}(1-\tau^*)}{2}\sum_{l=1}^{i-1}z_l^2(t-\tau(t))\cr
&+\frac{(i-1)e^{\tau}}{2(1-\tau^*)}\tilde{\eta}_i^2 z_i^2.
\end{align}
Substituting (\ref{equA2})-(\ref{equA4}) to (\ref{equA1}) leads
\begin{align*}
&z_i(f_i-\sum_{j=1}^{i-1}\frac{\partial\alpha_{i-1}}{\partial x_j}(x_{j+1}+f_j))\cr
\leq &\sum_{l=1}^{i-1}z_l^2+\frac{e^{-\tau}(1-\tau^*)}{2}\sum_{l=1}^{i}z_l^2(t-\tau(t))
+\Omega_iz_i^2,
\end{align*}
where $\Omega_i=\frac{(i-1)}{2}\tilde{\eta}_i^2
+\frac{(i-1)}{2(1-\tau^*)}e^{\tau}\tilde{\eta}_i^2
+\bar{\eta}_i+\frac{1}{2(1-\tau^*)}e^{\tau}\bar{\eta}_i^2$
is a positive constant.
\end{proof}

\section{The Proof of (\ref{equ572})}
\renewcommand{\theequation}{B.\arabic{equation}}\setcounter{equation}{0}
\begin{proof}
It follows from \textbf{H2} and (\ref{equ51}) that
\begin{align}\label{equB1}
&z_ig_i\xi_i-z_i\sum_{j=1}^{i-1}\frac{\partial\alpha_{i-1}}{\partial x_j}g_j\xi_j\cr
\leq &\frac{1}{2}|z_i|(g_i^2+|\xi_i|^2)
+\frac{1}{2}|z_i|\sum_{j=1}^{i-1}\Xi_{j(i-1)}(g_j^2+|\xi_j|^2)\cr
\leq & |z_i|\zeta_i(\sum_{l=1}^i(|x_l|+|x_l(t-\tau(t))|)\cr
&+\sum_{j=1}^{i-1}\sum_{l=1}^j(|x_l|+|x_l(t-\tau(t))|))\cr
&+\frac{1}{2}|z_i|(|\xi_i|^2+\sum_{j=1}^{i-1}\Xi_{j(i-1)}|\xi_j|^2)\cr
=& |z_i|\zeta_i\sum_{l=1}^{i}(i-l+1)(|x_l|+|x_l(t-\tau(t))|)\cr
&+\frac{1}{2}|z_i|(|\xi_i|^2+\sum_{j=1}^{i-1}\Xi_{j(i-1)}|\xi_j|^2)\cr
\leq& |z_i|\bar{\zeta}_i\sum_{l=1}^{i}(|x_l|+|x_l(t-\tau(t))|)\cr
&+\frac{1}{2}|z_i|(|\xi_i|^2+\sum_{j=1}^{i-1}\Xi_{j(i-1)}|\xi_j|^2)\cr
\leq& |z_i|\bar{\zeta}_i(|z_1|+|z_1(t-\tau(t))|
+\sum_{l=2}^i(|z_l|+\beta_{l-1}|z_{l-1}|)\cr
&+\sum_{l=2}^i(|z_l(t-\tau(t))|+\beta_{l-1}|z_{l-1}(t-\tau(t))|))\cr
&+\frac{1}{2}|z_i|(|\xi_i|^2+\sum_{j=1}^{i-1}\Xi_{j(i-1)}|\xi_j|^2)\cr
\leq &|z_i|\tilde{\zeta}_i(\sum_{l=1}^{i-1}(|z_l|+|z_l(t-\tau(t))|)\cr
&+|z_i|\bar{\zeta}_{i}(|z_i|+|z_i(t-\tau(t))|)\cr
&+\frac{1}{2}|z_i|(|\xi_i|^2+\sum_{j=1}^{i-1}\Xi_{j(i-1)}|\xi_j|^2),
\end{align}
where
$\zeta_i=\max_{1\leq j\leq i-1}\{\frac{1}{2}\bar{\theta}, \frac{1}{2}\bar{\theta}\Xi_{j(i-1)}\}$,
$\bar{\zeta}_i=\zeta_ii$,
$\tilde{\zeta}_{i}=\bar{\zeta}_{i}\max_{1\leq l\leq i-1}\{1+\beta_l\}$,
$i=2, 3, \cdots, n$.

By Young's inequality, one obtains
\begin{align}
\label{equB2}
\frac{1}{2}|z_i||\xi_i|^2 \leq& \frac{1}{16\pi_i}z_i^2+\pi_i|\xi_i|^4,\\
\label{equB3}
|z_i|\tilde{\zeta}_i\sum_{l=1}^{i-1}|z_l|
\leq& \frac{(i-1)}{2}\tilde{\zeta}_i^2 z_i^2+\sum_{l=1}^{i-1}\frac{1}{2}z_l^2,\\
\label{equB4}
\frac{1}{2}|z_i|\sum_{j=1}^{i-1}\Xi_{j(i-1)}|\xi_j|^2
\leq& \sum_{j=1}^{i-1}\frac{1}{16\pi_{ij}}\Xi_{j(i-1)}^2z_i^2\cr
&+\sum_{j=1}^{i-1}\pi_{ij}|\xi_j|^4,\\
\label{equB5}
|z_i|\bar{\zeta}_{i}|z_i(t-\tau(t))|
\leq& \frac{e^{-\tau}(1-\tau^*)}{2}z_i^2(t-\tau(t))\cr
&+\frac{e^{\tau}}{2(1-\tau^*)}\bar{\zeta}_i^2 z_i^2, \\
\label{equB6}
|z_i|\tilde{\zeta}_i\sum_{l=1}^{i-1}|z_l(t-\tau(t))|
\leq&\frac{e^{-\tau}(1-\tau^*)}{2}\sum_{l=1}^{i-1}z_l^2(t-\tau(t))\cr
&+\frac{(i-1)e^{\tau}}{2(1-\tau^*)}\tilde{\zeta}_i^2 z_i^2.
%\label{equB6}
%\frac{1}{2}|z_i|\sum_{j=1}^{i-1}\Xi_{j(i-1)}|\xi_j|^2
%\leq& \sum_{j=1}^{i-1}\frac{1}{16\pi_{ij}}\Xi_{j(i-1)}^2z_i^2\cr
%&+\sum_{j=1}^{i-1}\pi_{ij}|\xi_j|^4.
\end{align}
Substituting (\ref{equB2})-(\ref{equB6}) to (\ref{equB1}) yields
\begin{align*}
&z_ig_i\xi_i-z_i\sum_{j=1}^{i-1}\frac{\partial\alpha_{i-1}}{\partial x_j}g_j\xi_j\cr
\leq &\sum_{l=1}^{i-1}z_l^2+\Lambda_iz_i^2+\pi_i|\xi_i|^4+\sum_{j=1}^{i-1}\pi_{ij}|\xi_j|^4\cr
&+\frac{e^{-\tau}(1-\tau^*)}{2}\sum_{l=1}^{i}z_l^2(t-\tau(t)),
\end{align*}
where $\Lambda_i=\frac{(i-1)}{2}\tilde{\zeta}_i^2
+\frac{(i-1)}{2(1-\tau^*)}e^{\tau}\tilde{\zeta}_i^2
+\bar{\zeta}_i+\frac{1}{2(1-\tau^*)}e^{\tau}\bar{\zeta}_i^2
+\frac{1}{16\pi_i}
+\sum_{j=1}^{i-1}\frac{1}{16\pi_{ij}}\Xi_{j(i-1)}^2$ is a positive constant.
\end{proof}

%\section*{Acknowledgment}


\ifCLASSOPTIONcaptionsoff
  \newpage
\fi

\end{document}